\documentclass[11pt]{article}

\usepackage[letterpaper,margin=1in]{geometry}
\usepackage[T1]{fontenc}
\usepackage{lmodern}
\usepackage{microtype}
\usepackage{authblk}

\usepackage{amsmath}
\usepackage{amssymb,amsthm,bm}
\usepackage{graphicx}
\usepackage{xcolor}
\usepackage{booktabs}
\usepackage{multirow}
\usepackage{diagbox}
\usepackage{soul}
\usepackage{setspace}
\usepackage[font=small,labelfont=bf,labelsep=period]{caption}
\usepackage[title]{appendix}

\usepackage[colorlinks=true]{hyperref}
\definecolor{linkblue}{rgb}{0.0,0.2,0.55}
\hypersetup{
  linkcolor=linkblue,
  citecolor=linkblue,
  urlcolor=linkblue,
  pdftitle={Covariance-Informed Ensemble Score Filtering: Embedding Localized Covariance Geometry into Score-Based Transport},
  pdfauthor={Siming Liang, Peter Jan van Leeuwen, Hristo G. Chipilski, Zezhong Zhang, Feng Bao, Guannan Zhang},
}
\usepackage{apacite}
\let\cite\shortcite   %

\usepackage{algorithm}
\usepackage{algpseudocode}

\makeatletter
\newcommand{\frontheading}[1]{%
  \begin{center}\@endparpenalty\@M
    \bfseries #1\par\nopagebreak\vspace{-.5em}%
  \end{center}\nopagebreak\@beginparpenalty\@M}
\newenvironment{frontblock}[1]
  {\small\frontheading{#1}\quotation}
  {\endquotation}
\newenvironment{plainlanguagesummary}
  {\begin{frontblock}{Plain Language Summary}}
  {\end{frontblock}}
\newcommand{\keypoints}[3]{%
  \begingroup\small\frontheading{Key Points}%
  \begin{list}{\textbullet}{%
      \setlength{\leftmargin}{\dimexpr\leftmargini+1.2em\relax}%
      \setlength{\rightmargin}{\leftmargini}%
      \setlength{\labelwidth}{0.8em}\setlength{\labelsep}{0.4em}%
      \setlength{\itemsep}{1pt}\setlength{\parsep}{0pt}}%
    \if\relax\detokenize{#1}\relax\else\item #1\fi
    \if\relax\detokenize{#2}\relax\else\item #2\fi
    \if\relax\detokenize{#3}\relax\else\item #3\fi
  \end{list}\par\endgroup}
\makeatother
\newcommand{\keywords}[1]{%
  \begin{quotation}\small\noindent\textbf{Keywords:} #1\end{quotation}}

\soulregister\eqref7
\soulregister\ref7
\soulregister\cite7

\newcommand{\R}{\mathbb{R}}
\newcommand{\E}{\mathbb{E}}
\newcommand{\Var}{\operatorname{Var}}

\newcommand{\diag}{\operatorname{diag}}

\newcommand{\N}{\mathcal{N}}
\newcommand{\M}{\mathcal{M}}
\newcommand{\Hcal}{\mathcal{H}}
\newcommand{\I}{\mathbf{I}}
\newcommand{\xbf}{\mathbf{x}^{f}}
\newcommand{\xba}{\mathbf{x}^{a}}
\newcommand{\xb}{\mathbf{x}}
\newcommand{\yb}{\mathbf{y}}
\newcommand{\db}{\mathbf{d}}
\newcommand{\eb}{\boldsymbol{\varepsilon}}
\newcommand{\vb}{\mathbf{v}}
\newcommand{\zb}{\mathbf{z}}
\newcommand{\ub}{\mathbf{u}}
\newcommand{\Cmat}{\mathbf{C}}
\newcommand{\Bmat}{\mathbf{B}}
\newcommand{\Dmat}{\mathbf{D}}
\newcommand{\Kmat}{\mathbf{K}}
\newcommand{\Lmat}{\mathbf{L}}
\newcommand{\Rmat}{\mathbf{R}}
\newcommand{\Hmat}{\mathbf{H}}

\newcommand{\Vmat}{\mathbf{V}}

\newcommand{\fb}{\mathbf{f}}
\newcommand{\scoreb}{\mathbf{s}}
\newcommand{\gb}{\mathbf{g}}
\newcommand{\xbarf}{\overline{\mathbf{x}}^{f}}
\newcommand{\ybarf}{\overline{\mathbf{y}}^{f}}

\theoremstyle{plain}
\newtheorem{proposition}{Proposition}

\theoremstyle{remark}

\begin{document}

\title{Covariance-Informed Ensemble Score Filtering: Embedding Localized Covariance Geometry into Score-Based Transport}

\author[1]{Siming Liang}
\author[2]{Peter Jan van Leeuwen}
\author[3]{Hristo G.~Chipilski}
\author[4]{\\Zezhong Zhang}%
\author[5]{Feng Bao}
\author[1]{Guannan Zhang}
\affil[1]{Computer Science and Mathematics Division, Oak Ridge National Laboratory}
\affil[2]{Department of Atmospheric Science, Colorado State University}
\affil[3]{Department of Scientific Computing, Florida State University}
\affil[4]{Department of Mathematics and Statistics, Auburn University}
\affil[5]{Department of Mathematics, Florida State University}
\date{}

\maketitle

\begin{abstract}
Background-error covariance models are central to atmospheric data assimilation because they spread partial observational information across space, variables, and dynamically coupled scales. Ensemble and variational methods exploit this covariance geometry efficiently, but their standard analysis formulations are largely built around Gaussian or locally quadratic approximations. Score-based ensemble filtering provides a framework for nonlinear distributional transport, but existing Ensemble Score Filter (EnSF) formulations lack an intrinsic covariance-mediated pathway for propagating likelihood information from observed to unobserved state components. We introduce a covariance-informed ensemble score filter (CI-EnSF) that embeds localized forecast-error covariance geometry directly into the forward diffusion and reverse transport. The proposed forward diffusion is covariance-shaped in standardized state space and becomes isotropic under the corresponding control-variable transformation. This covariance structure continuously propagates observational information from observed to unobserved state components throughout the reverse transport analysis trajectory. Experiments with Lorenz--96 and surface quasi-geostrophic (SQG) dynamics show a clear regime dependence. For linear observations, a well-tuned LETKF is more accurate, while CI-EnSF maintains stable cycling and effective covariance-mediated information propagation. Under strongly nonlinear observations and under climatological initialization, CI-EnSF achieves lower analysis error and greater cycling robustness than the Gaussian ensemble benchmarks such as the LETKF, with improvements extending across spatial scales in the nonlinear SQG experiment. These results suggest that the principal benefit of CI-EnSF arises from combining nonlinear distributional transport with the mature covariance geometry used in conventional data assimilation.
\end{abstract}

\keypoints
{Localized covariance geometry is embedded directly into ensemble score filter's reverse transport dynamics for nonlinear data assimilation}
{Cross covariances continuously spread partial observation information to unobserved state components during the score-based transport}
{Nonlinear data assimilation shows improved accuracy, robustness, and stability under challenging observation nonlinearities}

\begin{plainlanguagesummary}
Weather and climate prediction systems combine model forecasts with incomplete observations to estimate the evolving state of the atmosphere and the ocean. Conventional methods use statistical relationships between model variables to spread information from observed locations to places and variables that are not directly measured, but their analysis steps are usually based on approximately Gaussian uncertainty. This study develops a filtering method that keeps this familiar covariance structure while allowing the ensemble to describe a non-Gaussian probability distribution. The covariance structure is built directly into a generative AI-driven transport so that sparse observations can influence unobserved parts of the state throughout the analysis. Tests with an idealized chaotic model and a two-layer surface turbulent flow show that the method retains conventional covariance-based behavior for linear observations and becomes more advantageous when the observation relationship is strongly nonlinear. The results provide a way to combine established data-assimilation covariance modeling with nonlinear probabilistic ensemble transport.
\end{plainlanguagesummary}


\section{Introduction}
\label{sec:intro}

Atmospheric data assimilation combines an imperfect model forecast with incomplete and noisy
observations to estimate the evolving state of the atmosphere \cite{AschEtAl2016,EvensenEtAl2022,KalnayEtAl2024}. The problem is intrinsically
underdetermined in the sense that the number of model degrees of freedom is generally far larger than the number of
observations available at a given analysis time, and many dynamically important variables are only
indirectly observed. Consequently, the ability to propagate information away from observation
locations is as important as the local fit to the observations themselves. In {traditional} variational and
ensemble data assimilation, this role is carried by the forecast- or background-error covariance
structure, which determines how an observed innovation modifies remote, unobserved and
multivariately coupled components of the state \cite{Lorenc1986,Bannister2008a,Bannister2008b}.

Modern atmospheric data-assimilation systems have therefore invested heavily in the representation
and application of covariance geometry. In variational methods, {the covariance is built around physics-informed estimates of spatial correlation length scales and multivariate dynamical balance relations, followed by extensive physics-guided tuning for optimal results.}
In ensemble Kalman filters, short-range forecast ensembles provide flow-dependent
covariances, {fine-tuned via localization and inflation to suppress} sampling artifacts and model errors arising from finite ensembles
\cite{HoutekamerMitchell1998,GaspariCohn1999,HamillEtAl2001,WhitakerHamill2002,Hunt2007,HoutekamerZhang2016}. Hybrid ensemble--variational methods
combine static and flow-dependent covariance information and have become an important route for
using finite ensembles without discarding mature variational covariance models
\cite{HamillSnyder2000,Lorenc2003,Buehner2005}. These developments provide a robust and physically interpretable
second-order geometry for high-dimensional atmospheric inference.

The same success also exposes a limitation. A covariance model, by construction, represents only second-order statistics. Forecast uncertainty generated by nonlinear dynamics can be skewed,
heavy-tailed or multimodal, and nonlinear observation operators can further produce non-Gaussian
analysis distributions that are used in subsequent forecasts \cite{BocquetEtAl2010,KurosawaPoterjoy2021}. In such regimes, forecast distributions with the same mean and covariance
can imply substantially different Bayesian updates. Ensemble Kalman and incremental variational
methods remain highly effective in many operational regimes, but their standard analysis steps do
not attempt to preserve the full shape of a non-Gaussian forecast distribution \cite{Chipilski2025}. This motivates a
complementary question: \emph{can mature atmospheric covariance models
be retained while the analysis itself is performed as nonlinear distributional transport?}

Score-based diffusion methods provide one possible route. Rather than estimating a single analysis
state or applying one Gaussian ensemble transform, a score-based filter evolves particles in an
artificial pseudo-time using gradients of log density. The Ensemble Score Filter
\cite<EnSF; >{Bao2024} was introduced as a training-free score-based method for high-dimensional nonlinear
filtering. It approximates forecast-distribution information from an ensemble and gradually
incorporates likelihood information along a reverse diffusion. The method is attractive for data
assimilation because it avoids offline neural score training and evaluates nonlinear observation
information at the evolving ensemble members. Subsequent EnSF developments have also highlighted a
difficulty that is especially important in geophysical applications: under partial observations,
the raw likelihood information directly constrains only observed variables. Image inpainting
\cite{Liang2025Inpainting} and a two-step linear-regression update \cite{Xiong2026EnSFLR} have
been proposed to reconstruct or update unobserved components. These studies demonstrate both the
promise of score-based assimilation and the importance of an explicit mechanism for
observed-to-unobserved information transfer.

The present work takes a different route. Instead of appending a separate reconstruction or
regression step after score-based updating, we embed data-assimilation covariance geometry directly
into the diffusion and reverse transport.
The proposed covariance-informed Ensemble Score Filter, termed CI-EnSF,
performs its artificial diffusion in standardized state coordinates, with the
localized forecast-error covariance geometry embedded directly into the diffusion
tensor. Under the corresponding control-variable transformation, the diffusion
becomes isotropic in control space while remaining shaped by the prescribed
forecast-error covariance geometry in physical space.
In this way, the same covariance structure used in conventional data
assimilation to spread observational information also governs the direction in which likelihood
information is transported by the score-based analysis.
This construction provides two important properties. First, the covariance structure can be
introduced without requiring an explicit covariance inversion in the prior contribution
to the reverse transport. The resulting formulation therefore preserves an analytical,
training-free prior update while retaining the familiar control-variable geometry of atmospheric
data assimilation. Second, cross-covariances act continuously during the reverse transport, so that
information introduced at observed locations can influence remote and unobserved state components
throughout the analysis trajectory rather than through a separate post-analysis reconstruction.
The method can therefore be viewed as embedding the information-spreading mechanism of ensemble
and variational data assimilation within a distributional posterior-transport framework. {The methodology has similarities with Particle Flow Filters, which aim to transport ensemble members directly from prior to posterior \cite{Hu_vanLeeuwen_2021,Yeh_vanLeeuwen_2026a,Yeh_vanLeeuwen_2026b}.} To ensure robust repeated cycling, we augment CI-EnSF
with two innovation-driven feedback mechanisms. The first adjusts the prior uncertainty scale using
a finite-ensemble-corrected innovation-consistency diagnostic {similar to inflation in Ensemble Kalman filtering}. The second tempers the likelihood
when innovations are large and automatically relaxes that tempering as the filter approaches a
well-calibrated cycling regime. Together, these mechanisms provide separate feedback on prior
spread and likelihood strength without requiring access to the true state.

The contributions of this paper are fourfold. First, we formulate ensemble score filtering in a
data-assimilation control-variable geometry, establishing a direct connection between isotropic
score transport in control space and covariance-shaped transport in physical space. Second, we
derive an analytical member-conditioned transition model in which the covariance operator drops
out of the prior transport while remaining active in the likelihood transport. Third, we
show how cross-covariance structure provides continuous propagation of observational information
from observed to unobserved variables during reverse score transport, thereby linking CI-EnSF to
the central information-spreading mechanism of ensemble and variational data assimilation. Fourth,
we develop innovation-based controllers for prior-spread calibration {to improve robustness in repeated cycling} and likelihood tempering to improve robustness {during spin up}.

The numerical study is organized around data-assimilation questions rather than model
complexity. Lorenz--96 experiments are designed to test consistency in a linear--Gaussian regime,
sparse observations, nonlinear observation operators and adaptive cycling. Surface
quasi-geostrophic experiments then examine spatially structured and cross-component information
propagation in a multiscale turbulent system. In particular, the SQG experiments test whether covariance-informed score transport remains effective in a high-dimensional, multiscale system and whether its relative advantage changes between linear and strongly nonlinear observation regimes.

\section{Data Assimilation Problem and Covariance Geometry}
\label{sec:problem}
The data-assimilation problem considered in this work is to infer a
high-dimensional model state from incomplete and possibly nonlinear observations.
Of particular interest is the common setting in which only a subset of the state
variables is observed through a nonlinear operator, so that an effective analysis must not only fit the
observations but also propagate their information to dynamically and {nonlinearly} related,
unobserved components. We first formulate this Bayesian filtering problem and then
identify the forecast-error covariance geometry that provides this
observation-to-state coupling in conventional data assimilation.

Let \(\xb_k\in\R^n\) denote the model state at analysis time \(t_k\). The discrete forecast model
and observation model are written as
\begin{equation}\label{eq:da}
\begin{aligned}
  \xb_k &= \M_{k}(\xb_{k-1}) + \boldsymbol{\eta}_k, \\
  \yb_k &= \Hcal_k(\xb_k) + \eb_k,
  \qquad \eb_k\sim\N(0,\Rmat_k),
\end{aligned}
\end{equation}
where
\(\boldsymbol{\eta}_k\) represents model error and \(\Rmat_k\) is the specified observation-error covariance. Given observations through time \(t_{k-1}\), the forecast density is denoted by
\(p_k^f(\xb)\). Bayes' rule defines the analysis density as
\begin{equation}
  p_k^a(\xb)
  \propto
  p(\yb_k\mid \xb)\,p_k^f(\xb).
  \label{eq:bayes}
\end{equation}
We represent \(p_k^f\) by a forecast ensemble of size $N_e$, i.e.,
\begin{equation}\label{eq:priorensemble}
  \bm X_k^f = \left\{\xbf_{k,i}\right\}_{i=1}^{N_e},
\end{equation}
and seek an analysis ensemble \(\bm X_k^a\) whose mean, spread and higher-order structure are
consistent with the posterior information represented by Eq.~\eqref{eq:bayes}.
For notational simplicity, we suppress the analysis-cycle subscript \(k\) in Eq.~\eqref{eq:priorensemble} throughout the remainder of the paper and denote the \(i\)th forecast ensemble member by $\xbf_{i}$.

The mechanism by which partial observations constrain the full state is most clearly seen through the forecast-error covariance. A conventional covariance-based analysis uses {either an offline-generated covariance matrix or an} ensemble forecast-error covariance:
\begin{equation}\label{eq:Bmat}
  {\Bmat}^f
  =
  \frac{1}{N_e-1}
  \sum_{i=1}^{N_e}
  \left(\xbf_i-\xbarf\right)
  \left(\xbf_i-\xbarf\right)^{\top},
  \qquad
  \xbarf
  =
  \frac{1}{N_e}\sum_{i=1}^{N_e}\xbf_i .
\end{equation}
where $\xbarf$ is the forecast ensemble mean. When having partial observations, the state vector $\xb$ and the covariance matrix $\Bmat^f$ can be partitioned as
\begin{equation}
  \xb=
  \begin{bmatrix}\xb_o\\ \xb_u\end{bmatrix},
  \qquad
  \Bmat^f=
  \begin{bmatrix}
    \Bmat_{oo} & \Bmat_{ou}\\
    \Bmat_{uo} & \Bmat_{uu}
  \end{bmatrix},
  \label{eq:blockB}
\end{equation}
where \(\xb_o\) and \(\xb_u\) denote observed and unobserved components,
and \(\Bmat_{uo}\) denotes the cross-covariance through which information is
transferred from \(\xb_o\) to \(\xb_u\). {The observed components are the components that are active in the observation operator $\Hcal_k(\xb_k)$.} Since the likelihood depends only
on \(\xb_o\), the gradient of the log likelihood is zero with respect to $\xb_u$, i.e.,
\begin{equation}
  \nabla_{\xb}\log p(\yb\mid\xb)
  =  \begin{bmatrix}\gb_o\\0\end{bmatrix},
\end{equation}
where $\gb_o$ denotes the
gradient of the log likelihood with respect to the observed state components. {When the prior is assumed to be Gaussian, as in variational and Ensemble Kalman methods, the likelihood forcing is proportional to}
\begin{equation}
  \Bmat^f\nabla_{\xb}\log p(\yb\mid\xb)
  =
  \begin{bmatrix}
    \Bmat_{oo}\gb_o\\
    \Bmat_{uo}\gb_o
  \end{bmatrix},
\label{eq:blockpropagation}
\end{equation}
where the cross-covariance \(\Bmat_{uo}\) is the direct pathway through which observed
information constrains unobserved variables.

Eq.~\eqref{eq:blockpropagation} {shows} that the covariance-based information spreading mechanism is essential to conventional data
assimilation. Ensemble Kalman and variational methods exploit the covariance information very effectively, but their standard analysis updates
are built around Gaussian or locally quadratic approximations and therefore
do not explicitly represent the full shape of a nonlinear posterior. In
contrast, score-based filtering provides a framework for nonlinear
distributional transport, but existing isotropic EnSF \cite{Bao2024} formulations do not
contain an intrinsic counterpart of the cross-covariance pathway in
Eq.~\eqref{eq:blockpropagation}. Under partial observations, additional
reconstruction or regression steps are therefore needed to update the unobserved variables. These complementary limitations motivate CI-EnSF. The goal is to retain the mature covariance geometry used for
observation-to-state information propagation while embedding it directly
into a score-based transport capable of representing nonlinear posterior structures.
\section{CI-EnSF: Covariance-Informed Ensemble Score Filtering}
\label{sec:cvtensf}
This section develops CI-EnSF by embedding forecast-error covariance geometry
directly into score-based ensemble transport. We first review the isotropic EnSF
formulation and its limitation under partial observations, and then introduce a
covariance-informed diffusion in standardized coordinates. The resulting formulation
retains the covariance geometry in the likelihood transport while allowing it to
cancel analytically from the prior transport.

\subsection{Overview of the isotropic ensemble score filter}

Score-based filtering introduces an artificial diffusion time \(t\in[0,1]\) that is distinct from
the physical model time. A forward stochastic differential equation (SDE) gradually transforms
the target distribution (e.g., the posterior distribution) into an easy-to-sample reference
distribution,
\begin{equation}
  d\zb_t = \fb(\zb_t,t)\,dt + \mathbf{G}(t)\,d\mathbf{W}_t,
  \label{eq:genericforward}
\end{equation}
with properly chosen drift and diffusion coefficients
\(\fb(\zb_t,t)\) and \(\mathbf{G}(t)\).
The forward SDE defines a family of marginal distributions connecting the target distribution (a.k.a, the prior distribution)
to the reference distribution. Generation or analysis sampling is then performed by initializing
from the reference distribution and integrating the associated reverse SDE:
\begin{equation}
  {d\zb_t}
  = \left[\fb(\zb_t,t)
    +\mathbf{G}(t)\mathbf{G}(t)^{\top}\scoreb_t(\zb_t)\right] dt + \mathbf{G}(t)\,d\overleftarrow{\mathbf{W}}_t,
  \label{eq:spfodegeneric}
\end{equation}
where $\overleftarrow{\mathbf{W}}_t$ is the reverse Brownian motion and $\scoreb_t(\zb_t)=\nabla\log p_t(\zb_t)$ is the score function. The Fokker-Planck equation associated with the reverse SDE can be written as
\begin{equation}\label{eq:FK}
  \frac{p_t(\zb_t)}{dt}
  = -\nabla \cdot \left[\fb(\zb,t)p_t(\zb_t)
    +\mathbf{G}(t)\mathbf{G}(t)^{\top} \scoreb_t(\zb) p_t(\zb_t)\right]
    + \frac12\mathbf{G}(t)\mathbf{G}(t)^{\top}\Delta p_t(\zb_t).
\end{equation}
Substituting $\nabla p_t(\zb_t) = p_t(\zb_t) \nabla \log p_t(\zb_t) = \scoreb_t(\zb_t)p_t(\zb_t)$ into Eq.~\eqref{eq:FK}, we can rewrite the Fokker-Planck equation as
\begin{equation}
  \frac{p_t(\zb_t)}{dt}
  = -\nabla \cdot\left[\fb(\zb_t,t)p_t(\zb_t)
    -\frac12 \mathbf{G}(t)\mathbf{G}(t)^{\top}\scoreb_t(\zb_t) p_t(\zb_t)\right],
\end{equation}
which corresponds to the following probability-flow ODE:
\begin{equation}
  {d\zb_t}
  =
  \left[\fb(\zb_t,t)
  -\frac12\mathbf{G}(t)\mathbf{G}(t)^{\top}\scoreb_t(\zb_t)\right] dt,
  \label{eq:pfodegeneric}
\end{equation}
with integration from the terminal distribution at \(t=1\) toward the data distribution at
\(t=0\). While the SDE (\ref{eq:spfodegeneric}) and the ODE (\ref{eq:pfodegeneric}) define different paths in state space, their corresponding probability density function evolve the same, which is what concerns us here.

The central computational task in solving the reverse ODE is to evaluate or approximate the
score function
\begin{equation}
  \scoreb_t(\zb)
  =
  \nabla\log p_t(\zb).
  \label{eq:diff_score}
\end{equation}
In filtering, the analysis score is constructed from forecast and likelihood information.
In the EnSF family, this is implemented through a pseudo-time combination of a forecast score
and an observation score,
\begin{equation}
  \scoreb_t^a(\zb)
  \approx
  \scoreb_t^f(\zb)
  + w(t)\,\scoreb_t^{\rm lik}(\zb),
  \label{eq:scoresplit}
\end{equation}
where \(\scoreb_t^f(\zb)\) is the forecast (prior) score,
\(\scoreb_t^{\rm lik}(\zb)\) is an approximation to the likelihood contribution induced by
the observation model in Eq.~\eqref{eq:da}, and \(w(t)\) controls how observation information
is introduced along the reverse transport path. The original EnSF is training-free: the forecast
score is estimated directly from the forecast ensemble rather than from a separately trained
neural score model \cite{Bao2024}. The method has demonstrated strong performance in solving several highly nonlinear, high-dimensional data-assimilation problems \cite{XiongEtAl2026ObservationNetworks, BaoEtAl2026UnitedFilter, HuynhEtAl2026AdaptiveLearning}.

For linear observation operators, the forecast-score and likelihood contributions can be
evaluated analytically under the Gaussian-mixture prior assumption. For nonlinear
observation operators, however, the exact pseudo-time likelihood score is generally expensive
to evaluate. The EnSF therefore uses an approximate likelihood score together with the
weighting function \(w(t)\) to control the relative contribution of prior and likelihood
information along the reverse trajectory. This approximation introduces posterior-score error,
and recent iterative EnSF method proposed in \cite{zhang2026iensf} explicitly seeks to reduce this
source of error.

The present work addresses a complementary challenge of isotropic EnSF: assimilation with
nonlinear and partial observations. Suppose the likelihood depends directly only on a subset
\(\xb_o\) of the physical state, with the remaining components denoted by \(\xb_u\). Then
the physical-space likelihood gradient has the block structure
\begin{equation}
  \nabla_{\xb}\log p(\yb\mid\xb)
  =
  \begin{bmatrix}
    \gb_o\\
    0
  \end{bmatrix},
\end{equation}
where
\(\gb_o=\nabla_{\xb_o}\log p(\yb\mid\xb_o)\).
Under isotropic score transport, there is no intrinsic covariance operator analogous to the
cross-covariance pathway identified in Eq.~\eqref{eq:blockpropagation} that transfers this
likelihood forcing from the observed to the unobserved components. This limitation has motivated
inpainting and regression-based EnSF extensions for partially observed systems
\cite{Liang2025Inpainting,Xiong2026EnSFLR}. These approaches can be effective, but they
introduce an additional reconstruction or regression mechanism to communicate observational
information to the unobserved state, rather than making this coupling an intrinsic part of the
score transport itself. CI-EnSF takes a different approach by embedding the forecast-error
covariance geometry directly into the diffusion tensor, so that covariance-mediated information
spreading acts continuously throughout the reverse transport.

\subsection{Standardized covariance representation}
\label{sec:standardization}

The limitation of isotropic EnSF identified in Section~3.1 motivates introducing
forecast-error covariance geometry directly into the score-based diffusion.
To separate componentwise forecast variance from cross-component correlation and
to remove physical units from the diffusion coordinates, we represent this
covariance geometry in a standardized state space.

The componentwise normalization matrix and the standardized forecast anomalies
are defined by
\begin{equation}
  \Dmat
=\diag(\sigma_1^f,\ldots,\sigma_n^f),
  \qquad
  \ub_i
  =
  \Dmat^{-1}(\xbf_i-\xbarf),
  \label{eq:standardize}
\end{equation}
where \(\sigma_j^f\) is the forecast standard deviation of state component \(j\).
For the numerical experiments, the correlation structure is estimated from the
standardized ensemble anomalies and then localized and regularized. Let
\(\mathbf{A}\in\mathbb{R}^{n\times N_e}\) denote the anomaly matrix whose
\(i\)th row is \(\ub_i^{\top}\). The localized and regularized correlation
matrix is defined by
\begin{equation}
  \Cmat
  =
  \frac{1}{1+\varepsilon}
  \left[
    \Cmat_{\rm loc}\circ
    \frac{\mathbf{A}\mathbf{A}^{\top}}{N_e-1}
    +
    \varepsilon\I
  \right],
  \label{eq:localizedC}
\end{equation}
where \(\circ\) denotes the Schur product,
\(\Cmat_{\rm loc}\) is a prescribed localization matrix with unit diagonal,
and \(\varepsilon \I\) is a numerical regularization.
The diagonal regularization is to prevent nearly singular covariance
directions and improve the numerical robustness of the covariance factorization.
Although the localized correlation matrix is positive definite under the
localization used here, finite-ensemble sampling could produce very small
eigenvalues. The regularization introduces a controlled eigenvalue floor while
preserving the unit diagonal of the correlation matrix. We set $\varepsilon = 10^{-6}$ in all numerical experiments.

The localized correlation matrix is factorized as
\begin{equation}
  \Cmat
  =
  \Lmat\Lmat^{\top}.
  \label{eq:Cfactor}
\end{equation}
The corresponding localized physical-space forecast-error covariance model and
its square-root factor are then defined by
\begin{equation}
  \Bmat^f
  =
  \Dmat\Cmat\Dmat,
  \qquad
  \mathcal{L}
  =
  \Dmat\Lmat,
  \qquad
  \Bmat^f
  =
  \mathcal{L}\mathcal{L}^{\top}.
  \label{eq:Bphysical}
\end{equation}
The factors \(\Lmat\) and \(\mathcal{L}\) therefore describe the same localized
covariance geometry in standardized and physical coordinates, respectively.
This representation allows the artificial diffusion to be formulated in
dimensionless standardized coordinates while retaining the localized
forecast-error covariance structure required for information propagation in
physical space.
The particular covariance construction in Eq.~\eqref{eq:localizedC} is not
essential to the CI-EnSF formulation. Static, ensemble-based, hybrid, or
operator-defined covariance models may also be used, provided that {the action of the covariance matrix on a vector and a factorization of that matrix to generate random vectors at psuedo time $t=1$ can be efficiently implemented.}

\subsection{Covariance-informed forward diffusion process}

The standardized covariance representation in Section~\ref{sec:standardization}
provides the geometry needed to extend an isotropic EnSF to covariance-informed
transport. CI-EnSF incorporates this geometry directly into the diffusion
coefficient, so that stochastic perturbations are isotropic in control space but correlated according to the localized forecast-error structure in
standardized state space.
Using the correlation factor \(\Lmat\) defined in
Eq.~\eqref{eq:Cfactor}, the covariance-informed forward diffusion is defined by
\begin{equation}
  d\zb_t
  =
  b(t)\zb_t\,dt
  +
  \sigma(t)\Lmat\,d\mathbf{W}_t,
  \label{eq:cvtforward}
\end{equation}
where \(\mathbf{W}_t\) is a standard \(n\)-dimensional Wiener process.
The instantaneous diffusion covariance in standardized coordinates is therefore
\(\sigma^2(t)\Cmat\). Under the transformation back to physical coordinates,
the corresponding diffusion covariance is \(\sigma^2(t)\Bmat^f\), so the
artificial diffusion inherits the localized forecast-error covariance geometry
introduced in Section~\ref{sec:standardization}.

The introduction of the covariance factor \(\Lmat\) into the forward SDE is not
an ad hoc modification of the isotropic EnSF. Similar covariance
factorizations are widely used in variational data assimilation through
control-variable transforms, where \(\Cmat=\Lmat\Lmat^{\top}\) maps an isotropic
control variable into a state with the prescribed covariance geometry. The
following proposition shows that the same interpretation applies to CI-EnSF.
\begin{proposition}[Control-space equivalence]
Assume that \(\Cmat\) is positive definite and that its factor
\(\Lmat\) is nonsingular. Under the control-variable transformation
\(\zb_t=\Lmat\vb_t\), the covariance-informed forward diffusion in
Eq.~\eqref{eq:cvtforward} is equivalent to the isotropic control-space
process
\begin{equation}
  d\vb_t
  =
  b(t)\vb_t\,dt
  +
\sigma(t)\,d\mathbf{W}_t.
  \label{eq:isotropicv}
\end{equation}
Thus, the same forward diffusion is isotropic in covariance-whitened control
coordinates and covariance-shaped in standardized state coordinates.
\end{proposition}

The equivalence follows directly from the pseudo-time independence of
\(\Lmat\) within one analysis cycle. Substitution of the transformation
\(\zb_t=\Lmat\vb_t\) into Eq.~\eqref{eq:cvtforward}, followed by
left multiplication with \(\Lmat^{-1}\), gives the isotropic process in
Eq.~\eqref{eq:isotropicv}. In the present implementation, the diagonal
regularization introduced in Eq.~\eqref{eq:localizedC} ensures a full-rank
correlation matrix. If a positive-semidefinite but low-rank covariance model is
used instead, the corresponding control-space representation applies on the
range of the covariance factor and requires an appropriate reduced-space
formulation.

The temporal evolution of the conditional mean and covariance is specified
through the mean-decay and variance-growth schedules
\begin{equation}
  \alpha_t=1-t,
  \qquad
  \beta_t^2=t.
  \label{eq:scheduleab}
\end{equation}
The corresponding drift and diffusion amplitudes are
\begin{equation}
  b(t)=\frac{1}{t-1},
  \qquad
  \sigma^2(t)=\frac{1+t}{1-t}.
  \label{eq:schedules}
\end{equation}

Under the schedules in Eqs.~\eqref{eq:scheduleab}--\eqref{eq:schedules},
the conditional transition law of the forward diffusion is
\begin{equation}
  \zb_t\mid\zb_0
  \sim
  \N\!\left(
    (1-t)\zb_0,\;
    t\Cmat
  \right).
  \label{eq:forwardkernel}
\end{equation}
To construct a member-conditioned representation of the forecast distribution,
the initial diffusion state associated with the standardized forecast member
\(\ub_i\) is assigned the following Gaussian distribution
\begin{equation}
  \zb_0\mid\ub_i
  \sim
  \N\!\left(
    \ub_i,\;
    \gamma\Cmat
  \right),
  \label{eq:memberprior}
\end{equation}
where \(\gamma>0\) is a dimensionless scaling parameter controlling the covariance
width around each forecast member.
The positive scale \(\gamma\) also serves as a regularization of the empirical
forecast distribution. At the ensemble level, the member-conditioned
construction can be interpreted as an equal-weight Gaussian mixture centered at
the forecast members, with each component having covariance \(\gamma\Cmat\).
If \(\gamma=0\), each component collapses to a Dirac delta function and its covariance
along the forward diffusion becomes \(t\Cmat\), which vanishes as \(t\to0\).
The corresponding prior score therefore becomes singular at $t=0$ and forces each member back toward its original forecast location, so that a finite likelihood forcing cannot produce a nonzero continuous
displacement in the limit \(t\to0\). A positive \(\gamma\) removes this
degeneracy, thereby allowing the
likelihood transport to produce a nontrivial analysis update.

\begin{proposition}[Analytical CI-EnSF kernel]
\label{prop2}
The combination of the forward transition law in
Eq.~\eqref{eq:forwardkernel} and the member-conditioned initial distribution
in Eq.~\eqref{eq:memberprior} gives the conditional transition kernel, i.e., the distribution of the intermediate state $\zb_t$ conditional on the forecast member $\ub_i$,
\begin{equation}
  {
  \zb_t\mid\ub_i
  \sim
  \N\!\left(
    (1-t)\ub_i,\;
    c(t)\Cmat
  \right),
  \qquad
  c(t)=(1-t)^2\gamma+t
  }.
  \label{eq:centralKernel}
\end{equation}
The conditional covariance therefore remains a scalar multiple of the same
localized correlation geometry throughout pseudo-time.
\end{proposition}

An important consequence of Proposition~\ref{prop2} is that the dependence on
both the forecast member and the prior-covariance scale vanishes at the terminal
time. The resulting terminal distribution is common to all ensemble members,
\begin{equation}
  \zb_1\mid\ub_i
  \sim
  \N(\mathbf{0},\Cmat),
\label{eq:terminaldistribution}
\end{equation}
which provides the reference distribution from which the reverse transport is
initialized. A complete proof of Proposition~\ref{prop2} is provided in
Appendix~\ref{sec:app1}.

\subsection{Cancellation of the covariance from the prior transport}

The analytical member-conditioned kernel derived in
Proposition~\ref{prop2} provides a closed-form expression for the forecast
score at every pseudo-time. For ensemble member \(i\), the corresponding
forecast score is
\begin{equation}
  \scoreb_{t,i}^f(\zb)
  =
  -\bigl(c(t)\Cmat\bigr)^{-1}
  \left[\zb-(1-t)\ub_i\right].
  \label{eq:priorscore}
\end{equation}
Although evaluation of this score in isolation appears to require the inverse
of the correlation matrix, the reverse transport depends on the score only
through its covariance-weighted form.
For the covariance-informed forward diffusion in
Eq.~\eqref{eq:cvtforward}, the associated probability-flow ODE takes the form
\begin{equation}
  \frac{d\zb}{dt}
  =
  b(t)\zb
  -
  \frac{1}{2}\sigma^2(t)\Cmat\scoreb_t(\zb).
  \label{eq:cvtpfode}
\end{equation}
The covariance weighting in this transport leads to the following
cancellation for the forecast-score contribution.

\begin{proposition}[Prior-transport cancellation]
For the member-conditioned kernel in Eq.~\eqref{eq:centralKernel}, the
covariance-weighted forecast score satisfies
\begin{equation}
  \Cmat\scoreb_{t,i}^f(\zb)
  =
  -
  \frac{\zb-(1-t)\ub_i}{c(t)}.
  \label{eq:cancelC}
\end{equation}
Consequently, evaluation of the forecast contribution to the probability-flow
velocity requires neither an explicit application of \(\Cmat^{-1}\) nor a
matrix function of \(\Cmat\).
\end{proposition}

The cancellation follows directly from the scalar-matrix structure of the
conditional covariance \(c(t)\Cmat\). In particular, the inverse covariance
appearing in the forecast score is exactly canceled by the covariance
preconditioning already present in the probability-flow ODE. This property
preserves an analytical and computationally inexpensive forecast contribution
even when a spatially structured covariance model is used.
Importantly, this cancellation does not remove the covariance geometry from
CI-EnSF. The correlation matrix still determines the {random draws at $t=1$,}
whose distribution is
$
  \zb_1
  \sim
  \N(\mathbf{0},\Cmat),
$
and more importantly for data assimilation, the same covariance geometry
remains explicitly active in the likelihood contribution to the reverse
transport. The next section shows that this remaining covariance weighting
provides the pathway through which observational information is propagated
from observed to unobserved state components.

\subsection{Covariance-informed likelihood transport}
\label{sec:likelihoodtransport}

The prior-transport cancellation established in Section~3.4 leaves the
forecast contribution analytically tractable while preserving the covariance
operator in the likelihood contribution. We now show that this remaining
covariance weighting provides the mechanism by which observational information
is transferred from observed to unobserved state components during the reverse
transport.
The physical state associated with a standardized coordinate \(\zb\) is defined by
\begin{equation}
  \xb(\zb)
  =
  \xbarf+\Dmat\zb,
  \label{eq:destandardize}
\end{equation}
where \(\xbarf\) is the forecast ensemble mean defined in
Eq.~\eqref{eq:Bmat}, and \(\Dmat\) is the normalization matrix defined
in Eq.~\eqref{eq:standardize}. For Gaussian observation errors, the physical-space
log likelihood is
\begin{equation}
  \log p(\yb\mid\xb)
  =
  -\frac12
  \left[\yb-\Hcal(\xb)\right]^{\top}
  \Rmat^{-1}
  \left[\yb-\Hcal(\xb)\right]
  +\mathrm{const}.
  \label{eq:loglikelihood}
\end{equation}
Let \(\mathbf{J}_{\Hcal}(\xb)\) denote the Jacobian of the observation operator.
The corresponding likelihood gradient in physical coordinates is
\begin{equation}
  \nabla_{\xb}\log p(\yb\mid\xb)
  =
  \mathbf{J}_{\Hcal}(\xb)^{\top}
  \Rmat^{-1}
  \left[\yb-\Hcal(\xb)\right].
  \label{eq:likphysical}
\end{equation}
Using the state transformation in Eq.~\eqref{eq:destandardize}, the chain rule
gives the exact likelihood score at the target endpoint \(t=0\),
\begin{equation}
  \scoreb^{\rm lik}_0(\zb)
  =
  \Dmat
  \mathbf{J}_{\Hcal}\!\left(\xb(\zb)\right)^{\top}
  \Rmat^{-1}
  \left[
    \yb-\Hcal\!\left(\xb(\zb)\right)
  \right].
  \label{eq:likstandard}
\end{equation}

For \(t>0\), the exact pseudo-time likelihood score induced by the forward
diffusion is generally not available in closed form for a nonlinear observation
operator. Following the original EnSF construction \cite{Bao2024,BaoEtAl2024ScoreFilter}, CI-EnSF
therefore approximates this quantity by evaluating the physical observation
likelihood at the evolving pseudo-state. The approximate likelihood score for
member \(i\) is defined as
\begin{equation}
  \widetilde{\scoreb}_{t,i}^{\rm lik}(\zb)
  =
  \Dmat
  \mathbf{J}_{\Hcal}\!\left(\xb(\zb)\right)^{\top}
  \Rmat^{-1}
  \left[
    \yb-\Hcal\!\left(\xb(\zb)\right)
  \right].
  \label{eq:approxlikscore}
\end{equation}
The corresponding approximate analysis score is
\begin{equation}
\widetilde{\scoreb}_{t,i}^{a}(\zb)
  =
  \scoreb_{t,i}^{f}(\zb)
  +
  w(t)\,\widetilde{\scoreb}_{t,i}^{\rm lik}(\zb),
  \qquad
  t\in[0,1].
  \label{eq:weightedscore}
\end{equation}
The likelihood-weighting function is chosen as
\(w(t)=1-t\). This choice provides endpoint consistency: at \(t=0\), the
approximate likelihood score reduces to the physical likelihood score in
Eq.~\eqref{eq:likstandard}, whereas at \(t=1\), the likelihood contribution
vanishes and the analysis and forecast diffusions share the same terminal
reference distribution.

Substitution of the analysis-score approximation in
Eq.~\eqref{eq:weightedscore} into the probability-flow ODE in
Eq.~\eqref{eq:cvtpfode}, together with the prior-transport cancellation in
Eq.~\eqref{eq:cancelC}, gives the member-wise analysis transport
\begin{equation}
  \frac{d\zb}{dt}
  =
  b(t)\zb
  -
  \frac12\sigma^2(t)
  \left[
    -
    \frac{\zb-(1-t)\ub_i}{c(t)}
    +
    w(t)\,\Cmat \;
\widetilde{\scoreb}_{t,i}^{\rm lik}(\zb)
  \right].
  \label{eq:analysis_ode}
\end{equation}
The role of the covariance geometry is most transparent when the reverse
integration direction is written using the reverse pseudo-time
\(\tau=1-t\). The likelihood contribution to the standardized-state velocity
then satisfies
\begin{equation}
  \left.
  \frac{d\zb}{d\tau}
  \right|_{\rm lik}
  =
  \frac12
  \sigma^2(t)w(t)
\Cmat\;\widetilde{\scoreb}_{t,i}^{\rm lik}(\zb).
  \label{eq:likvelocityz}
\end{equation}
Using the physical-state transformation in Eq.~\eqref{eq:destandardize} and
the covariance representation in Eq.~\eqref{eq:Bphysical}, the corresponding
physical-space likelihood velocity is
\begin{equation}
  \left.
  \frac{d\xb}{d\tau}
  \right|_{\rm lik}
  =
  \frac12
  \sigma^2(t)w(t)
  \Bmat^f
  \mathbf{J}_{\Hcal}(\xb)^{\top}
  \Rmat^{-1}
  \left[
    \yb-\Hcal(\xb)
  \right].
  \label{eq:likvelocityx}
\end{equation}

The likelihood transport in Eq.~\eqref{eq:likvelocityx} has the same
covariance-weighted observation-to-state geometry identified in
Eq.~\eqref{eq:blockpropagation}. In particular, when the likelihood depends
directly only on the observed state components, the cross-covariance block
\(\Bmat_{uo}\) transfers the likelihood forcing to the unobserved components.
Thus, covariance-mediated information spreading is not appended as a separate
reconstruction step in CI-EnSF; instead, it acts continuously as an intrinsic
part of the score-based reverse transport. This property constitutes the main
conceptual distinction between CI-EnSF and isotropic EnSF formulations that
require an additional inpainting or regression procedure for partially
observed systems.

\subsection{Relation to continuous-time Gaussian data assimilation}
\label{sec:classicalrelation}

The covariance-weighted likelihood transport derived in
Section~\ref{sec:likelihoodtransport} has a direct connection to the
continuous-time limit of Gaussian data assimilation. For a linear observation operator, \(\Hcal(\xb)=\Hmat\xb\), the Jacobian is
constant, \(\mathbf{J}_{\Hcal}(\xb)=\Hmat\). The reverse-time likelihood velocity therefore reduces to the
covariance-weighted form
\begin{equation}
  \left.
  \frac{d\xb}{d\tau}
  \right|_{\rm lik}
  \propto
  \Bmat^f\Hmat^\top\Rmat^{-1}
  [\yb-\Hmat\xb].
  \label{eq:cvtinstant}
\end{equation}
The covariance-weighted mapping in
Eq.~\eqref{eq:cvtinstant} differs in form from a discrete Kalman analysis,
whose gain contains the inverse of the innovation covariance
\(\Hmat\Bmat^f\Hmat^{\top}+\Rmat\). Because CI-EnSF incorporates likelihood
information continuously along the reverse pseudo-time trajectory, its
instantaneous likelihood velocity is more naturally compared with the
continuous-time Kalman--Bucy observation update.

The connection can be seen by considering a continuous observing system over
an infinitesimal physical-time interval \(\Delta t\). Let \(\Rmat_t\) denote
the continuous-time observation-noise covariance rate. The observation
averaged over the interval has the corresponding discrete observation-error
covariance
$
  \Rmat_d
  =
  \frac{\Rmat_t}{\Delta t}.
  \label{eq:Rd_main}
$
The discrete Kalman gain associated with this observation covariance is
\begin{equation}
  \Kmat_{\rm discrete}
  =
  \Delta t\,
  \Bmat^f\Hmat^{\top}
  \left(
    \Delta t\,\Hmat\Bmat^f\Hmat^{\top}
    +\Rmat_t
  \right)^{-1}.
  \label{eq:Kdiscrete_main}
\end{equation}
The gain per unit physical time therefore has the continuous-time limit
\begin{equation}
  \Kmat_t
  =
  \lim_{\Delta t\to0}
  \frac{\Kmat_{\rm discrete}}{\Delta t}
  =
  \Bmat^f\Hmat^{\top}\Rmat_t^{-1}.
  \label{eq:kb_limit_main}
\end{equation}
Thus, the CI-EnSF likelihood transport and the Kalman--Bucy observation update
share the same instantaneous covariance-weighted observation-to-state
geometry. In both cases, the factor \(\Bmat^f\Hmat^{\top}\) determines how
observation-space information is mapped into the model state, including the
propagation of information to unobserved components through forecast-error
cross-covariances. A detailed discrete-to-continuous derivation is provided in
Appendix~\ref{app:kalmanbucy}.

\section{Innovation-based CI-EnSF calibration for robust cycling}
\label{sec:adapt}

The covariance-informed transport supplies the desired DA geometry, but long cycling also requires
that forecast uncertainty and likelihood strength remain statistically compatible with the observed
innovations. We distinguish two feedback problems. First, the ensemble may become too narrow or too
broad relative to its realized forecast errors. Second, a broad forecast can produce large innovations during spin-up, and the
approximate likelihood forcing used in the reverse transport can make the
explicitly integrated ODE numerically stiff. We address these problems
with separate, complementary controls.

\subsection{Innovation-space consistency diagnostic}
\label{sec:rho}
Let the forecast ensemble be mapped into observation space,
\begin{equation}
  \yb_i^f = \Hcal(\xbf_i),
  \qquad
  \ybarf = \frac{1}{N_e}\sum_{i=1}^{N_e}\yb_i^f,
\end{equation}
and define the innovation of the forecast ensemble mean,
$
  \db = \yb-\ybarf.
$
For observation component \(j\), we define the ensemble-predicted observation
variance
\begin{equation}
  s_{f,j}^2
  =
  \frac{1}{N_e-1}
  \sum_{i=1}^{N_e}
  \left(y_{i,j}^f-\bar y_j^f\right)^2.
  \label{eq:obsvariance}
\end{equation}
We first establish the consistency target used below. Let
$
  \nu^2_{f,j}
  =
  \Var\!\left[\Hcal_j(\xb^f)\right]
$
denote the population forecast variance of observation component \(j\).
Under the standard idealized ensemble-consistency argument, the truth and
the \(N_e\) forecast ensemble members are independent draws from the same
forecast distribution. Let the observation be
$
  y_j
  =
  y_j^{\rm true}+\varepsilon_j \text{ with }
  \E[\varepsilon_j]=0,
  \Var(\varepsilon_j)=r_j,
$
where the observation error is independent of the forecast ensemble and truth.
Writing \(\mu_j=\E[\Hcal_j(\xb^f)]\), the innovation can be decomposed as
\begin{equation}
  d_j
  =
  \left(y_j^{\rm true}-\mu_j\right)
  -
  \left(\bar y_j^f-\mu_j\right)
  +
  \varepsilon_j.
\end{equation}
The three terms on the right-hand side are assumed to be mutually independent with variances
\(\nu^2_{f,j}\), \(\nu^2_{f,j}/N_e\), and \(r_j\), respectively.
Hence, we have
\begin{equation}
    \E[d_j]=0
  \quad \text{ and }\quad
  \E[d_j^2]
  =
  \left(1+\frac{1}{N_e}\right)\nu^2_{f,j}+r_j,
  \label{eq:innovation_consistency}
\end{equation}
where the factor \(1+1/N_e\) accounts for the sampling error of the finite-ensemble
mean relative to an independent truth realization.

Eq.~\eqref{eq:innovation_consistency} provides a direct statistical
consistency condition. In particular, if the population variances
\(\nu_{f,j}\) were known, the normalized innovation statistic
\begin{equation}
  \rho_\star
  =
  {\displaystyle\sum_{j=1}^{N_y}d_j^2}\Bigg/
       {\displaystyle\sum_{j=1}^{N_y}
       \left[
       \left(1+\frac{1}{N_e}\right)\nu^2_{f,j}+r_j
       \right]}
  \label{eq:rhostar}
\end{equation}
would satisfy
\begin{equation}
  \E[\rho_\star]=1.
  \label{eq:rhostar_expectation}
\end{equation}
In practice, the population forecast variances \(\nu^2_{f,j}\) are unavailable
and are replaced by their ensemble estimates \(s_{f,j}^2\). We therefore
use the following computable consistency diagnostic
\begin{equation}
  {
  \rho
  =
  {\displaystyle\sum_{j=1}^{N_y} d_j^2}\Bigg /
       {\displaystyle\sum_{j=1}^{N_y}\left[
         \left(1+\frac{1}{N_e}\right)s_{f,j}^2+r_j
       \right]}.}
  \label{eq:rho}
\end{equation}
Because the denominator in Eq.~\eqref{eq:rho} contains finite-ensemble
variance estimates, \(\E[\rho]\) is not exactly equal to one for finite
\(N_e\). Nevertheless, \(s_{f,j}^2\) is an unbiased estimator of
\(\nu^2_{f,j}\), so Eq.~\eqref{eq:rho} is the natural plug-in form of
Eq.~\eqref{eq:rhostar}, and a statistically calibrated ensemble is
expected to produce \(\rho\) fluctuating around unity. Persistent
\(\rho>1\) indicates that the realized innovation energy exceeds that
predicted by the forecast ensemble and specified observation errors,
consistent with an underdispersive forecast ensemble in observation space.
Conversely, persistent \(\rho<1\) indicates that the predicted total
innovation variance is larger than the realized innovation energy.

The ratio in Eq.~\eqref{eq:rho} is formed after summing over
observations rather than by averaging componentwise normalized innovations.
This aggregation reduces sensitivity to sampling fluctuations in individual
ensemble-variance estimates and avoids giving excessive weight to
components with anomalously small denominators.
Because the resulting innovation statistic is still noisy from one analysis cycle
to the next, we apply exponential moving average approach.
Specifically,
\(\rho_{\rm smooth}^{(k)}\) is updated as a weighted average of the current
diagnostic and its previous smoothed value,
\begin{equation}
  \rho_{\rm smooth}^{(k)}
  =
  (1-\zeta)\rho_{\rm smooth}^{(k-1)}
  +
  \zeta\rho^{(k)},
  \label{eq:rhosmooth}
\end{equation}
where \(\zeta\in(0,1]\) trades sampling noise against feedback lag. We set $\zeta = 0.3$ in all numerical experiments.

\subsection{Adaptive prior-covariance scaling}
\label{sec:adaptivegamma}

The parameter \(\gamma\) in Eq.~\eqref{eq:memberprior} determines the width of the
member-conditioned prior at \(t=0\). A larger \(\gamma\) produces a broader analysis ensemble and,
after model propagation, generally increases forecast spread. We therefore use \(\rho_{\rm smooth}\) as a feedback indicator to guide the adjustment of $\gamma$. Specifically, after $(k-1)$-th assimilation cycle, we use the following formula the adjust the scaling parameter $\gamma_k$ used for the $k$-th cycle,
\begin{equation}
  \gamma_k= \min\Big\{\max\left(\sqrt{\rho_{\rm smooth}} \gamma_{k-1}, \gamma_{\min}\right), \gamma_{\max}\Big\},
  \label{eq:gammacontrol}
\end{equation}
where \(\gamma_{\rm max}\) and \(\gamma_{\rm min}\) are tunable hyper-parameters determining the upper and lower bounds of $\gamma_k$, respectively.  Eq.~\eqref{eq:gammacontrol} provides a simple multiplicative feedback mechanism for dynamically
adjusting the prior-covariance scale. When
\(\rho_{\rm smooth}>1\), the realized innovation variance is larger than that
predicted by the ensemble, and the factor \(\sqrt{\rho_{\rm smooth}}\) increases
\(\gamma\) for the next analysis cycle, thereby inflating the
prior spread. Conversely, when \(\rho_{\rm smooth}<1\), the same
update decreases \(\gamma\) and reduces the prior spread. The square-root
dependence provides a moderate response to departures from consistency, while
the bounds \(\gamma_{\min}\) and \(\gamma_{\max}\) prevent unrealistically small
or large prior scales. Thus, the controller continuously steers the prior
uncertainty toward the innovation-consistency target
\(\rho_{\rm smooth}\approx 1\).

A more direct innovation-based estimate of a multiplicative forecast-covariance
correction can be obtained from the scalar consistency relation
\(d^2=\alpha_B\sigma_b^2+\sigma_o^2\), giving
\(\widehat{\alpha}_B=(d^2-\sigma_o^2)/\sigma_b^2\).
Related innovation-based covariance estimation and adaptive-inflation
approaches are well established in data assimilation
\cite{DeeDaSilva1999,DesroziersEtAl2005,Anderson2007,Anderson2009,ElGharamti2018}.
Here we instead use the normalized total-innovation statistic
\(\rho\) from Section~\ref{sec:rho} as a feedback signal rather than as a
direct estimate of the inflation factor.
This feedback controller is designed to provide a more gradual adjustment of
the prior-covariance scale, which is desirable in the present
implementation because overly aggressive increases in \(\gamma\) can increase
the stiffness of the reverse ODE and thereby impose a more restrictive
numerical stability requirement. A systematic comparison between this
consistency-based feedback strategy and direct innovation-based estimation of
the covariance-inflation factor is left for future work.

\subsection{Innovation-driven likelihood tempering}
\label{sec:annealing}
Likelihood tempering, including its implementation through repeated
assimilation or inflation of the observation-error covariance, is a standard
strategy for moderating overly strong updates in nonlinear ensemble and
particle data assimilation
\cite{SakovEtAl2012,EmerickReynolds2013,
VanLeeuwenEtAl2019,KurosawaPoterjoy2021}.
In CI-EnSF, the same idea has an additional numerical role: because the
likelihood contribution enters the reverse ODE through
\(\Rmat^{-1}\), tempering reduces both the magnitude and local sensitivity of
the likelihood forcing and thereby alleviates ODE stiffness during
spin-up.

Adaptive \(\gamma\) alone does not directly control the likelihood-induced
stiffness of the reverse ODE during spin-up. In particular, a broad climatological ensemble can produce
large innovations and therefore {ODE stiffness during
spin-up.} We therefore temper the likelihood during
spin-up by introducing an innovation-dependent effective observation covariance.

For diagonal observation errors with comparable units within an observation
group, we define
$
  \overline{d^2}
  =
  \frac{1}{N_y}\sum_{j=1}^{N_y} d_j^2,
$
and
\begin{equation}
  \Rmat_{\rm eff}
  =
  \diag_{j=1,\ldots,N_y}
  \left\{
    \max\!\left(r_j,\kappa\overline{d^2}\right)
  \right\},
  \label{eq:Reff}
\end{equation}
where $r_j$ is the $j$-th diagonal entry of the observation error covariance matrix $\mathbf{R}$.
Here \(\kappa\ge0\) controls the tempering strength, which is the a hyper-parameter that needs to be tuned for each problem. For heterogeneous
observing systems, Eq.~\eqref{eq:Reff} should be applied separately by
observation type or after appropriate normalization.
When the innovations are large, the tempering rule increases selected
diagonal entries of the observation-error covariance, so that
\(\Rmat_{\rm eff}\succeq\Rmat\) and consequently weakens the likelihood
forcing through \(\Rmat_{\rm eff}^{-1}\).
This limits the stiffness encountered during spin-up. As the
filter approaches a statistically equilibrated cycling regime and the innovation variance decreases, Eq.~\eqref{eq:Reff} automatically returns
\(\Rmat_{\rm eff}\) toward the specified physical covariance \(\Rmat\), so that
the tempering becomes inactive without a prescribed cycle-dependent schedule.

\subsection{Complete CI-EnSF analysis cycle}
\label{sec:algorithm}
Algorithm~\ref{alg:cvtensf} summarizes one CI-EnSF analysis cycle. The presentation uses the
probability-flow ODE; a reverse-SDE implementation follows the same covariance geometry with the
corresponding stochastic term. The global hyper-parameters are omit in Algorithm \ref{alg:cvtensf} for notational simplicity.

\begin{algorithm}[h!]
\caption{One CI-EnSF analysis cycle}
\label{alg:cvtensf}
\begin{spacing}{1.3}
\begin{algorithmic}[1]
\Require Forecast \(X^f=\{\xbf_i\}_{i=1}^{N_e}\), observations \(\yb\), and \(\gamma\), \(\rho_{\rm smooth}\) from pervious cycle.
\State Compute \(\yb_i^f=\Hcal(\xbf_i)\), \(\ybarf\), \(\db=\yb-\ybarf\), and
\(s_{f,j}^2=\Var_i(y_{i,j}^f)\)

\State Compute \(\rho\) from Eq.~\eqref{eq:rho} and update
\(\rho_{\rm smooth}\) using Eq.~\eqref{eq:rhosmooth}

\State Update \(\gamma\) using Eq.~\eqref{eq:gammacontrol}

\State Construct \(\Rmat_{\rm eff}\) from Eq.~\eqref{eq:Reff}

\State Standardize \(X^f\) to obtain \(\ub_i\) and \(\Dmat\), and
estimate/localize \(\Cmat\)

\State Construct a covariance factor \(\Lmat\) satisfying
\(\Cmat=\Lmat\Lmat^{\top}\)

\For{\(i=1,\ldots,N_e\)}

  \State Draw terminal state \(\zb_i(1)=\Lmat\boldsymbol{\xi}_i\), with
  \(\boldsymbol{\xi}_i\sim\N(\mathbf{0},\I)\)

  \State If using perturbed observations, draw
  \(\yb_i=\yb+\eb_i\), with
  \(\eb_i\sim\N(\mathbf{0},\Rmat_{\rm eff})\)

  \State Integrate from \(t=1\) to \(t=t_{\rm stop}\):

  \State \hspace{1em} Compute \(c(t)=(1-t)^2\gamma+t\)

  \State \hspace{1em} Compute
  \(\Cmat\scoreb_{t,i}^{f}
  =-[\zb_i-(1-t)\ub_i]/c(t)\)

  \State \hspace{1em} Compute
  \(\xb_i(t)=\xbarf+\Dmat\zb_i(t)\)

  \State \hspace{1em} Evaluate
  \(\widetilde{\scoreb}_{t,i}^{\rm lik}\) from
  Eq.~\eqref{eq:approxlikscore} using \(\Hcal\), its Jacobian action,
  and \(\Rmat_{\rm eff}\)

  \State \hspace{1em} Integrate Eq.~\eqref{eq:analysis_ode} using
  \(w(t)=1-t\)

  \State Return
  \(\xba_i=\xbarf+\Dmat\zb_i(t_{\rm stop})\)

\EndFor

\State \Return \(X^a=\{\xba_i\}_{i=1}^{N_e}\), updated \(\gamma\), and updated
\(\rho_{\rm smooth}\)

\end{algorithmic}
\end{spacing}
\end{algorithm}

\section{Numerical experiments}
\label{sec:numerics}

The numerical experiments demonstrate the performance of the proposed CI-EnSF under a
progression of increasingly demanding regimes. The experiments using a 40-dimensional Lorenz--96 model provide a controlled setting in which
we first examine sparse linear observations, including the propagation of information into
unobserved state components, and then introduce a strongly nonlinear observation operator and a
climatological initial ensemble. A two-layer surface quasi-geostrophic (SQG) model then tests CI-EnSF in a $d=128\times 128 \times 2=32,768$ dimensional space. The purpose of this hierarchy is not to
show that CI-EnSF outperforms existing covariance-based methods in every setting. Instead, it is to
identify when covariance-informed score transport reproduces the behavior expected from standard
DA and when nonlinear distributional transport provides an additional benefit.

\subsection{Experimental design and evaluation metrics}
\label{sec:exp-design}
The main reference method is the local ensemble transform Kalman filter (LETKF), which performs a localized deterministic ensemble Kalman update in
ensemble space using the forecast ensemble mean and observation-space perturbations. Localization and ensemble-spread parameters are tuned
independently for each experiment. RML-3DVar is included in the Lorenz--96 experiments as a Gaussian, member-conditioned reference: each forecast member
defines the center of a Gaussian prior with the same localized covariance construction used by CI-EnSF, and the corresponding nonlinear variational problem is solved by Gauss--Newton iteration.

Algorithmic parameters are tuned and evaluated using independent realizations of the underlying dynamical system.
Localization and ensemble-spread parameters are tuned jointly when they interact, and the reported
results use only settings selected from the independent tuning realizations. This separation is important because the
optimal localization and inflation parameters differ substantially between warm-start initialization and
cold-start initialization, i.e., spin-up from a climatological ensemble.

The primary evaluation metrics are analysis root-mean-square error (RMSE) and the
Spread/RMSE ratio. At each analysis cycle, RMSE is computed from the analysis
ensemble mean over all state components and then averaged over the evaluation
period and across realizations. Ensemble spread is computed from the sample
variance across ensemble members and averaged over state components. For each
realization, the Spread/RMSE ratio is defined as the root-mean-square ensemble
spread over the evaluation period divided by the corresponding root-mean-square
analysis error, and the reported value is then averaged across realizations.

\subsection{Lorenz--96 with linear observations}
\label{sec:exp-gaussian}

The first experiment tests whether CI-EnSF performs well in a regime
where conventional Gaussian DA is already effective. The Lorenz--96 system has $d=40$ and
$F=8$ and is integrated with fourth-order Runge--Kutta. The truth uses a time step of $\Delta t=0.01$ and
the forecast model uses a coarser step of $\Delta t=0.02$, which introduces a small model-error component.
Analyses are performed every $\Delta t_{\rm obs}=0.2$ (24 h in the standard Lorenz--96 scaling),
and every fourth state variable is observed directly with independent Gaussian error
$\sigma_{\rm obs}=0.1$. The ensemble contains $N_e=20$ members initialized as the ground-truth trajectory plus a Gaussian noise from
$\mathcal{N}(0,0.25\mathbf{I})$. Each realization contains 40 analysis cycles, with verification statistics computed over cycle 10 to cycle 39. Five independent realizations are reserved for parameter tuning and another ten independent realizations for evaluation. After parameter fine tuning, the selected parameters are
$L=14$ with inflation 1.05 for LETKF, $\gamma=1.5$ and $L=3$ for RML-3DVar, and
$\gamma=0.05$ and $L=4$ for CI-EnSF.

\subsubsection{Baseline DA cycling performance}

Table~\ref{tab:gauss24} summarizes the evaluation-period statistics, while Figure~\ref{fig:gauss24} shows the evolution of the analysis error during cycling
and the RMSE ratio of observed and unobserved state components.
In terms of analysis accuracy, LETKF provides the strongest baseline, with a mean RMSE
of 0.076 over the ten independent evaluation realizations. CI-EnSF has a mean RMSE
of 0.1152, compared with 0.1208 for the covariance-matched RML-3DVar reference.
The difference between CI-EnSF and RML-3DVar is only 4.7\% and is not statistically
distinguishable in the paired comparison across realizations. Figure~\ref{fig:gauss24}a
shows the same ordering throughout the evaluation period: all three methods enter a
statistically stationary cycling regime, with LETKF maintaining the lowest analysis error
and CI-EnSF closely tracking RML-3DVar. Thus, replacing the member-conditioned
Gaussian analysis by covariance-informed score transport does not yield an accuracy
advantage for this case, but neither does it introduce unstable or qualitatively
different cycling behavior.

The ensemble-spread diagnostics provide a complementary assessment of calibration.
For an ensemble of $N_e=20$, the finite-ensemble reference value of the spread-to-RMSE ratio is
$\sqrt{N_e/(N_e+1)}=0.976$. The corresponding ratios are 1.07 for LETKF,
1.01 for CI-EnSF, and 0.76 for RML-3DVar. CI-EnSF is therefore close to the
finite-ensemble consistency reference, with modest overdispersion, whereas RML-3DVar
is appreciably underdispersive. This result is useful because the objective of CI-EnSF
is not only to obtain an accurate ensemble mean but also to maintain an ensemble whose
dispersion remains comparable with the analysis error during DA cycling.

Sparse observations provide a second baseline diagnostic. Only every fourth state
component enters the likelihood directly, so the remaining 75\% of the state must be
constrained through the forecast-error covariance. As shown in Table \ref{tab:gauss24}, errors are larger at unobserved components for all three methods.
Nevertheless, Figure~\ref{fig:gauss24}b shows that the unobserved-to-observed RMSE
ratio remains bounded and of order unity rather than
growing with cycling. This diagnostic by itself does not establish the mechanism of information transfer, but it demonstrates that CI-EnSF maintains control of state components that are absent from the likelihood.

Taken together, the linear-observation experiment establishes the intended baseline:
a well-tuned LETKF remains more accurate in the regime for which Gaussian ensemble
updating is well matched to the observation model, whereas CI-EnSF exhibits stable
cycling, near-consistent ensemble dispersion, and sustained control of unobserved state
components. The more demanding nonlinear experiments below therefore test whether
distributional transport provides additional benefit when the assumptions underlying
this linear baseline are progressively weakened.

\begin{table}[h!]
\centering
\small
\caption{Baseline performance for Lorenz--96 assimilation with partial linear observations. Statistics are computed over cycles 10--39 of ten
independent evaluation realizations. Overall, observed, and unobserved RMSE are
computed from the analysis ensemble mean and then averaged over the evaluation
period and across realizations. The Spread/RMSE ratio compares ensemble dispersion
with the realized analysis error. For $N_e=20$, the finite-ensemble consistency
reference is $\sqrt{N_e/(N_e+1)}=0.976$.
LETKF provides the lowest analysis error in this
linear-observation regime, while CI-EnSF remains close to the covariance-matched RML-3DVar reference, maintains near-consistent ensemble spread, and keeps errors in
unobserved state components bounded during repeated cycling.}
\label{tab:gauss24}
\begin{tabular}{ccccc}
\toprule
Method & Overall RMSE &  Observed RMSE & Unobserved RMSE & Spread/RMSE \\
\midrule
CI-EnSF   & 0.1152  & 0.0803 & 0.1240 & 1.01 \\
LETKF      & 0.0760  & 0.0581 & 0.0804 & 1.07 \\
RML-3DVar  & 0.1208  & 0.0763 & 0.1316 & 0.76 \\
\bottomrule
\end{tabular}
\end{table}
\begin{figure}[h!]
\centering
\includegraphics[width=\textwidth]{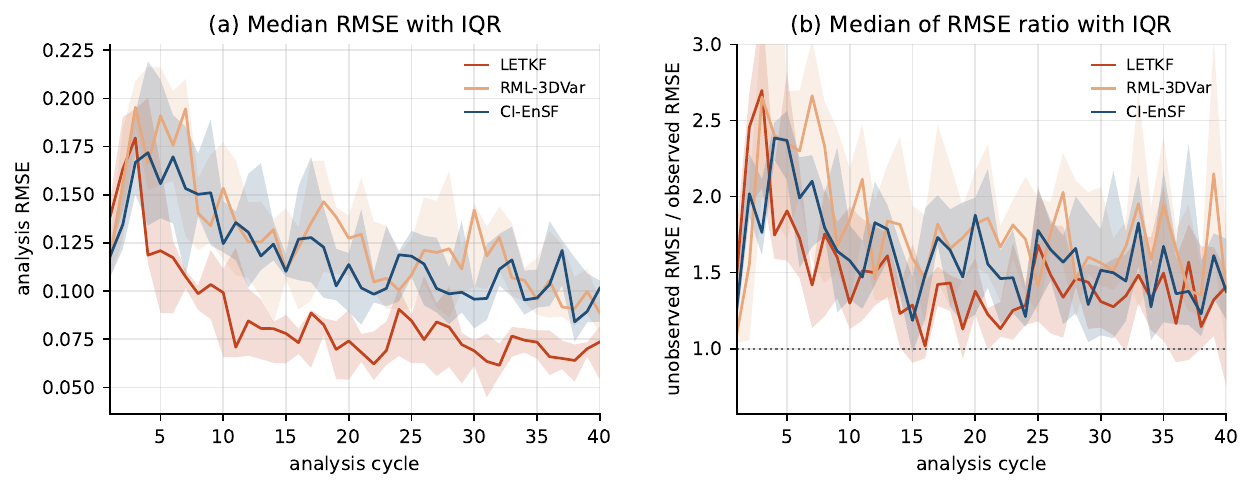}
\vspace{-0.8cm}
\caption{Cycling behavior for the Lorenz--96 linear-observation baseline. Statistics at each analysis cycle are computed across the ten independent evaluation realizations; solid curves denote the median and shading denotes the interquartile range (IQR). (a) Analysis RMSE of the ensemble mean. The
evaluation period is cycles 10--39, over which all three methods reach statistically stationary cycling, with LETKF maintaining the lowest error and CI-EnSF closely
tracking RML-3DVar. (b) Ratio of RMSE at the unobserved 75\% of state components
to that at the directly observed 25\% of state components; the horizontal reference at one denotes equal
error in the two subsets. CI-EnSF does not improve upon a well-tuned LETKF for linear observations, but it cycles stably and keeps unobserved-state errors at the same order as observed-state errors, consistent with effective covariance-mediated information spreading.}
\label{fig:gauss24}
\end{figure}

\subsubsection{Covariance-informed reverse transport during a single analysis}\label{sec:mechanism}
Figure~\ref{fig:mechanism} illustrates the mechanism responsible for updating unobserved variables within CI-EnSF.
For this mechanism diagnostic, we select one cycle of one evaluation realization using the pre-specified rule of the
largest forecast error over the unobserved components. We follow four neighboring states:
$x_{20}$ is observed and $x_{21}$--$x_{23}$ are unobserved. Starting from the same forecast ensemble
and the same terminal noise, the reverse ODE is solved with (i) 25\% observations and the localized
correlation $\Cmat$, (ii) the same 25\% observations with $\Cmat=\I$ (the isotropic EnSF \cite{Bao2024}),
and (iii) 100\% observations with $\Cmat=\I$ as a full-information reference. The first two cases
therefore isolate the role of cross-state covariance in the likelihood transport.
With the isotropic EnSF and only 25\% observations, the likelihood contribution is zero in the
three unobserved components and their ensemble-mean trajectories are nearly vertical. With the
localized covariance, the same observation at $x_{20}$ produces coherent motion in
$x_{21}$--$x_{23}$ toward the truth. The fully observed isotropic case shows the directions obtained
when those components are constrained directly. By comparing the first and third columns of Figure \ref{fig:mechanism}, we see that CI-EnSF bends the reverse-ODE trajectories of the unobserved state variables in directions similar to those obtained by EnSF with 100\% observations, demonstrating the effectiveness of the covariance structure in propagating observational information to unobserved components.

\begin{figure}[htbp]
  \centering  \includegraphics[width=0.82\textwidth]{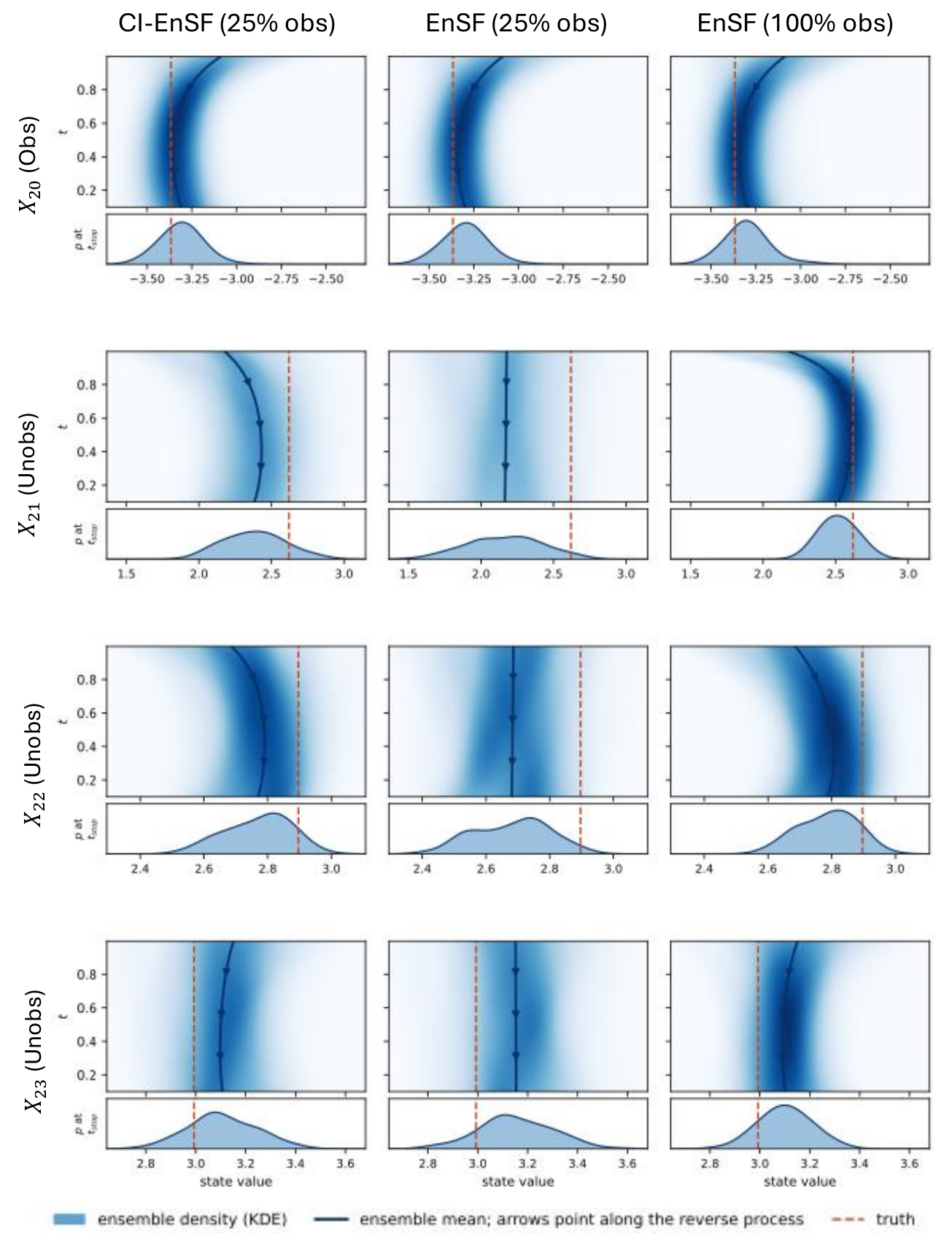}
  \caption{Mechanism of observed-to-unobserved information propagation in one Lorenz--96 analysis.
  The four rows show one observed state, $x_{20}$, and three unobserved neighbors,
  $x_{21}$--$x_{23}$, during one cycle of one evaluation realization. All columns start from the same forecast ensemble
  and terminal noise. The left column uses 25\% observations with the localized covariance
  $\Cmat$ (CI-EnSF); the center uses the same observations with $\Cmat=\I$ (isotropic EnSF); the
  right uses $\Cmat=\I$ with 100\% observations as a full-information reference. Pseudo-time runs
  from $t=1$ to $t_{\rm stop}=0.1$; shading denotes the density of 200 trajectories, the solid curve
  is the ensemble mean, the dashed line is the truth, and the lower strip gives the final analysis
  density. With sparse observations, cross-covariance is the pathway
  that transfers likelihood information into unobserved variables; removing it leaves the unobserved states essentially unaffected by the observation, whereas CI-EnSF moves them in the
  same direction as having full observation.}
  \label{fig:mechanism}
\end{figure}

\subsection{Lorenz--96 with nonlinear observations}
\label{sec:exp-arctan}

The second Lorenz--96 experiment tests CI-EnSF in handling DA problems with a nonlinear observation operator,
$h(x)=\arctan(x)$, and a climatological initial condition. The Lorenz--96 model, the ensemble size $N_e=20$, and
24-h analysis interval are unchanged, but 50\% of the state is observed with
$\sigma_{\rm obs}=0.05$. Each initial ensemble is formed from effectively independent states drawn
from a long free model integration, so the filter begins without state-specific information about the truth. Each assimilation realization contains 80 cycles; verification statistics are computed over cycles 40--79
after the initial transient. Hyper parameters are tuned using 20 independent realizations and evaluated
on 40 disjoint realizations. The selected settings are $L=2$ with inflation
1.15 for LETKF; $\gamma=0.02$, $L=1.5$, and five Gauss--Newton iterations for RML-3DVar; and
$\gamma_{\min} = 0.05$, $\gamma_{\max} = 0.15$  $\gamma_0=(\gamma_{\min}+\gamma_{\max})/2$, $L=1.5$, $N_t=400$, adaptive prior scaling, and likelihood tempering with
$\kappa=0.025$ for CI-EnSF.

\subsubsection{Accuracy, calibration, and cycling stability}\label{sec:nonlinear_acc}
Figure~\ref{fig:arctan24} and Table~\ref{tab:arctan24}
summarize the performance of the three methods under the nonlinear observation
operator and climatological initialization. In contrast to the partial linear-observation
experiment of Section~5.2, this configuration requires the filters both to spin up from
a broad prior ensemble and to assimilate observations through the saturating operator
$h(x)=\arctan(x)$. All three methods reduce the large initial error during the first
several analysis cycles, demonstrating that each can extract useful information from the
nonlinear observations. Their long-term cycling behavior, however, differs substantially.
\begin{table}[h!]
\centering
\small
\caption{Lorenz--96 cycling with the nonlinear observation operator $h(x)=\arctan(x)$ and climatological initialization. Forty
independent evaluation realizations are verified over cycles 40--79. RMSE (stable) and Spread/RMSE are computed
from realizations that remain stable through the experiment; RMSE (all) retains the large errors from numerically unstable
or filter-divergent realizations and is reported as a robustness diagnostic. The finite-ensemble reference
spread/RMSE is 0.976. In contrast to the linear reference experiment, CI-EnSF is both the most accurate method on stable realizations and the only method stable in all 40 evaluation realizations, showing that its advantage emerges under strong observation nonlinearity and difficult spin-up.}
\label{tab:arctan24}
\begin{tabular}{ccccc}
\toprule
Method & Stable realizations & RMSE (stable) & Spread/RMSE & RMSE (all) \\
\midrule
CI-EnSF & 40/40 & 0.3959 & 1.02 & 0.3959 \\
LETKF & 29/40 & 0.5315 & 1.19 & 1.1542 \\
RML-3DVar & 37/40 & 1.2186 & 0.98 & 1.2831 \\
\bottomrule
\end{tabular}
\end{table}
\begin{figure}[h!]
\centering
\includegraphics[width=\textwidth]{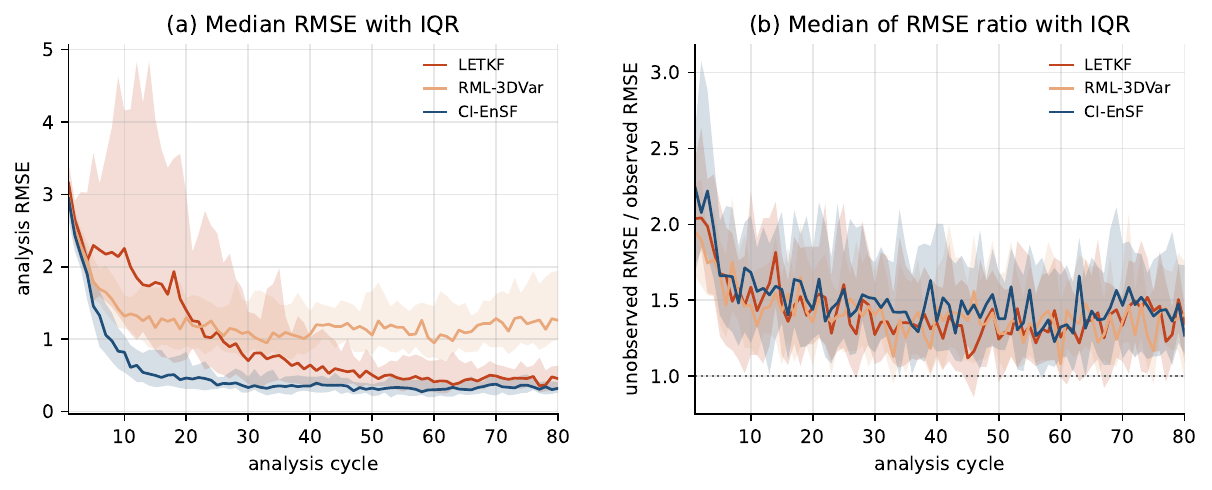}
\vspace{-0.8cm}
\caption{Spin-up and DA cycling under nonlinear Lorenz--96 observations. (a) Median analysis RMSE versus analysis cycle,
with the interquartile range (IQR) across realizations that remain finite at each cycle; the evaluation period is
cycles 40--79. (b) Ratio of RMSE at the unobserved 50\% of state components to that at
the directly observed 50\% of state components. All methods reduce the climatological initial error, but CI-EnSF reaches the lowest sustained RMSE while remaining stable in all 40 realizations; LETKF and
RML-3DVar become unstable or divergent in 11 and 3 realizations, respectively.}
\label{fig:arctan24}
\end{figure}

CI-EnSF reaches the lowest sustained analysis error after spin-up. Over the evaluation
period, its mean RMSE is 0.3959, compared with 0.5315 for LETKF and 1.2186 for
RML-3DVar among the realizations that remain stable. The median trajectories in
Figure~\ref{fig:arctan24}a show that this ordering is established during spin-up and
persists throughout the evaluation period rather than arising from a small number of
individual cycles. Relative to LETKF, the mean RMSE of CI-EnSF is reduced by about
30\% on the respective stable realizations, while the difference relative to RML-3DVar
is substantially larger. This behavior contrasts with the linear-observation experiment,
where LETKF was the most accurate method.

The ensemble-spread statistics show that the improved accuracy is not obtained by
collapsing the ensemble. The spread/RMSE ratios are 1.02 for CI-EnSF, 1.19 for
LETKF, and 0.98 for RML-3DVar, compared with the finite-ensemble consistency
reference of 0.976 for $N_e=20$. Thus, CI-EnSF and LETKF are moderately
overdispersive in this difficult nonlinear regime, whereas RML-3DVar is closer to the
nominal spread/RMSE reference. The latter result should not be interpreted as superior
filter performance, however, because RML-3DVar attains that ratio while maintaining
a much larger analysis error. The spread diagnostic therefore complements, rather than
replaces, the RMSE comparison.

Figure~\ref{fig:arctan24}b provides an additional diagnostic of partial-observation
behavior. After the initial transient, the ratio of unobserved to observed RMSE remains
of order unity for all three methods. The three curves exhibit broadly similar values
during the evaluation period, indicating that the lower total RMSE of CI-EnSF is not
produced by improving only the directly observed components. Instead,
the reduction in analysis error is maintained across both observed and unobserved states. This is consistent with the covariance-based information-spreading
mechanism illustrated in Figure~\ref{fig:mechanism}.

The clearest distinction among the methods is cycling stability. CI-EnSF remains
stable in all 40 independent evaluation realizations, whereas LETKF remains stable in
29 and RML-3DVar in 37. Because conditioning the RMSE only on successful
realizations can favor methods that lose the more difficult cases, Table~\ref{tab:arctan24}
also reports an RMSE summary over all realizations. Under this measure, the difference
widens substantially: CI-EnSF remains at 0.3959, while LETKF and RML-3DVar
increase to 1.1542 and 1.2831, respectively. The accuracy and stability diagnostics
indicate that CI-EnSF not only achieves a lower sustained analysis error, but does so without sacrificing realizations during repeated nonlinear
DA cycling.

\subsubsection{Sensitivity of CI-EnSF to prior scaling and likelihood tempering}
\label{sec:exp-adapt}
The nonlinear experiment uses two adaptive procedures: the innovation-based prior scale $\gamma$ and the
likelihood-tempering covariance $\Rmat_{\rm eff}$. Since $\mathbf{R} = 0.05^2 \mathbf{I}$, we can re-write Eq.~\eqref{eq:Reff} as $\Rmat_{\rm eff} = r_{\rm eff}\mathbf{I}$ with $r_{\rm eff} = \max(0.05^2,\kappa\overline{d^2})$. Thus, we use $r_{\rm eff}$ to represent $\Rmat_{\rm eff}$ in this test.
The roles of $\gamma$ and $r_{\rm eff}$ are separated with the existing four-configuration sensitivity experiment in which adaptive $\gamma$ and likelihood tempering are independently enabled or disabled while all other settings, initial ensembles, and observation sequences are held fixed.
Table~\ref{tab:adapt2x2} shows that likelihood tempering is the primary source of spin-up stability:
with fixed $\gamma$ it increases the number of stable realizations from 36/40 to 40/40, and with adaptive $\gamma$
it increases the number from 38/40 to 40/40. Adaptive $\gamma$ has a different role. Among the realizations
that remain stable without tempering, it reduces RMSE from 0.4343 to 0.3911. With both controls active, CI-EnSF achieves 40/40 stability and RMSE 0.3959.
Figure~\ref{fig:adapt2} makes the separation in time clear: likelihood tempering acts almost entirely during the climatological spin-up and relaxes to the physical
observation error, whereas $\gamma$ evolves more gradually with the innovation consistency diagnostic.
\begin{table}[htbp]
\centering
\small
\caption{Sensitivity of nonlinear Lorenz--96 cycling to adaptive prior scaling and
likelihood tempering. All four configurations are the same as that for CI-EnSF in Section \ref{sec:nonlinear_acc}; only the adaptive inflation parameter $\gamma$ and the threshold
$\kappa=0.025$ for likelihood tempering are switched on or off. Evaluation-period mean $\gamma$ and
${r_{\rm eff}}$ diagnose the action of the two controls over cycles 40--79.
Likelihood tempering is essential for robust climatological spin-up, while adaptive $\gamma$ primarily changes ensemble calibration and can
increase numerical stiffness when used without tempering.}
\label{tab:adapt2x2}
\begin{tabular}{ccccccc}
\toprule
$\gamma$ & tempering & stable realizations & RMSE (stable) & spread/RMSE & $\overline{\gamma}$ & ${\overline{r}_{\rm eff}}$ \\
\midrule
fixed    & off & 36/40 & 0.4343 & 1.03 & 0.100 & 0.0025 \\
adaptive & off & 38/40 & 0.3911 & 1.03 & 0.076 & 0.0025 \\
fixed    & on  & 40/40 & 0.3902 & 1.07 & 0.100 & 0.0028 \\
adaptive & on  & 40/40 & 0.3959 & 1.02 & 0.075 & 0.0029 \\
\bottomrule
\end{tabular}
\end{table}

\begin{figure}[htbp]
\centering
\includegraphics[width=\textwidth]{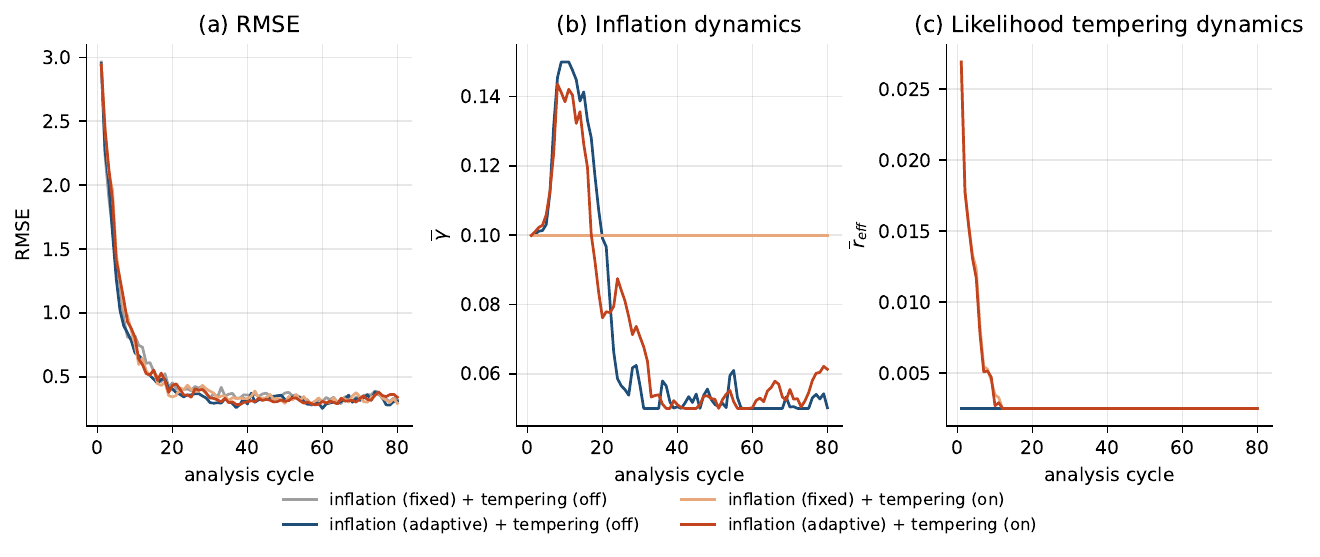}
\vspace{-0.8cm}
\caption{Time evolution of the two adaptive procedures in the nonlinear Lorenz--96 experiment, shown as
medians over 40 evaluation realizations. (a) Analysis RMSE during spin-up for the four combinations of
fixed/adaptive $\gamma$ and likelihood tempering; the vertical line marks the beginning of the
evaluation period. (b) Adaptive prior inflation scalar $\gamma$.
(c) Effective observation-error variance ${r_{\rm eff}}$. Likelihood tempering is a transient spin-up mechanism that rapidly relaxes to the specified observation error, whereas the adaptive prior-scaling procedure acts more gradually; robust cycling requires these adaptive procedures to be
consistent with the numerical stiffness of the reverse ODE.}
\label{fig:adapt2}
\end{figure}

\subsection{High-dimensional surface quasi-geostrophic model}
\label{sec:exp-sqg}

The SQG experiments test whether the Lorenz--96 conclusions persist in a spatially extended,
multiscale turbulent system. The model carries potential temperature $\theta$ on two bounding
surfaces at $z=0$ and $z=H=10$ km over a doubly periodic $20\,000$ km square. Each surface has
$128^2$ grid points (156 km spacing), giving $d=32,768$. The Rossby radius is
$L_R=1000$ km.
Because the circulation induced by a disturbance decays vertically over a depth proportional to its horizontal scale, the two surfaces interact at scales of order $L_R$ and larger, and act nearly independently at smaller scales.
$N_e=16$ ensemble members are used, so the state
dimension exceeds the ensemble size by more than three orders of magnitude. Each realization begins
from a climatological ensemble constructed from a long free model integration. Three independent realizations are reserved for parameter tuning and kept disjoint from the ten evaluation realizations.

\subsection{Linear observations}
\label{sec:sqg_linear}
In this linear observation setting, analyses are performed every 6 h for 100 cycles. At each cycle, 10\% of the grid points on both surfaces are selected
randomly and observed directly with \(\sigma_{\rm obs}=1\) K; the observation
network is redrawn at the next cycle. Verification statistics are computed over
cycles 50--99. After hyperparameter tuning, both methods use a localization
length of \(L=1000\) km.
With only two bounding surfaces, no separate vertical localization is applied, so that localization is a function of horizontal distance alone. LETKF uses RTPS \(=0.5\), whereas CI-EnSF uses
adaptive \(\gamma\) with $\gamma_{\min} = 0.05$, $\gamma_{\max} = 0.5$, and likelihood tempering with \(\kappa=0.025\). The
likelihood tempering is effectively inactive during the evaluation period
because \(\sqrt{r_{\rm eff}}\) has already relaxed to the specified
observation-error standard deviation of 1 K.

Both methods remain stable in all ten independent evaluation realizations.
Their ensemble dispersion is also close to the finite-ensemble consistency
reference,
$
\sqrt{\frac{N_e}{N_e+1}}
=
\sqrt{\frac{16}{17}}
\approx 0.970.
$
The mean Spread/RMSE ratios are 0.96 for CI-EnSF and 1.02 for LETKF,
indicating that both filters maintain ensemble spread commensurate with their realized analysis errors. The principal difference is analysis accuracy. LETKF achieves a mean RMSE of 1.2553 K, compared with 1.5033 K for CI-EnSF, so that the CI-EnSF error is 19\% larger in this linear-observation regime.
Figure~\ref{fig:sqg128} provides complementary spatial and cycling
diagnostics. Starting from a climatological ensemble, both methods rapidly reduce the initial error and subsequently maintain bounded errors throughout the assimilation period. At the final cycle of the representative realization, both analyses recover the dominant large-scale SQG structures. The remaining analysis errors are concentrated primarily along sharp fronts and filamentary features, where small displacement or amplitude errors produce locally large differences from the truth. The error fields produced by the two methods are qualitatively similar, although LETKF exhibits smaller overall error, consistent with the aggregate statistics. The cycling diagnostic further shows that both methods maintain control of the much larger unobserved portion of the state despite directly observing only 10\% of the grid points at each analysis time.

This experiment therefore provides an important high-dimensional
linear-observation baseline. Increasing the state dimension to
\(d=32{,}768\), more than three orders of magnitude larger than the ensemble size, does not by itself create an advantage for score-based transport. When the observation operator is linear and a well-tuned covariance-based update is well matched to the problem, LETKF remains more accurate, while CI-EnSF nevertheless exhibits stable cycling, near-consistent ensemble dispersion, and effective covariance-informed estimate of sparsely observed states.
\begin{table}[h!]
\centering
\small
\caption{Performance of CI-EnSF and LETKF in the high-dimensional SQG experiment with linear observations. RMSE denotes the mean analysis RMSE across realizations, while worst RMSE is the largest realization-mean RMSE among the ten evaluation realizations. Both methods remain stable and near-consistent in spread, while the tuned
LETKF achieves lower analysis error in this linear-observation regime.}
\label{tab:sqg128}
\begin{tabular}{ccccc}
\toprule
Method & Stable realizations & RMSE (K) & Worst RMSE (K) & Spread/RMSE \\
\midrule
CI-EnSF & 10/10 & 1.4671 & 1.6028 & 0.98 \\
LETKF & 10/10 & 1.2553 & 1.3452 & 1.03 \\
\bottomrule
\end{tabular}
\vspace{-0cm}
\end{table}
\begin{figure}[h!]
\centering
\includegraphics[width=0.95\textwidth]{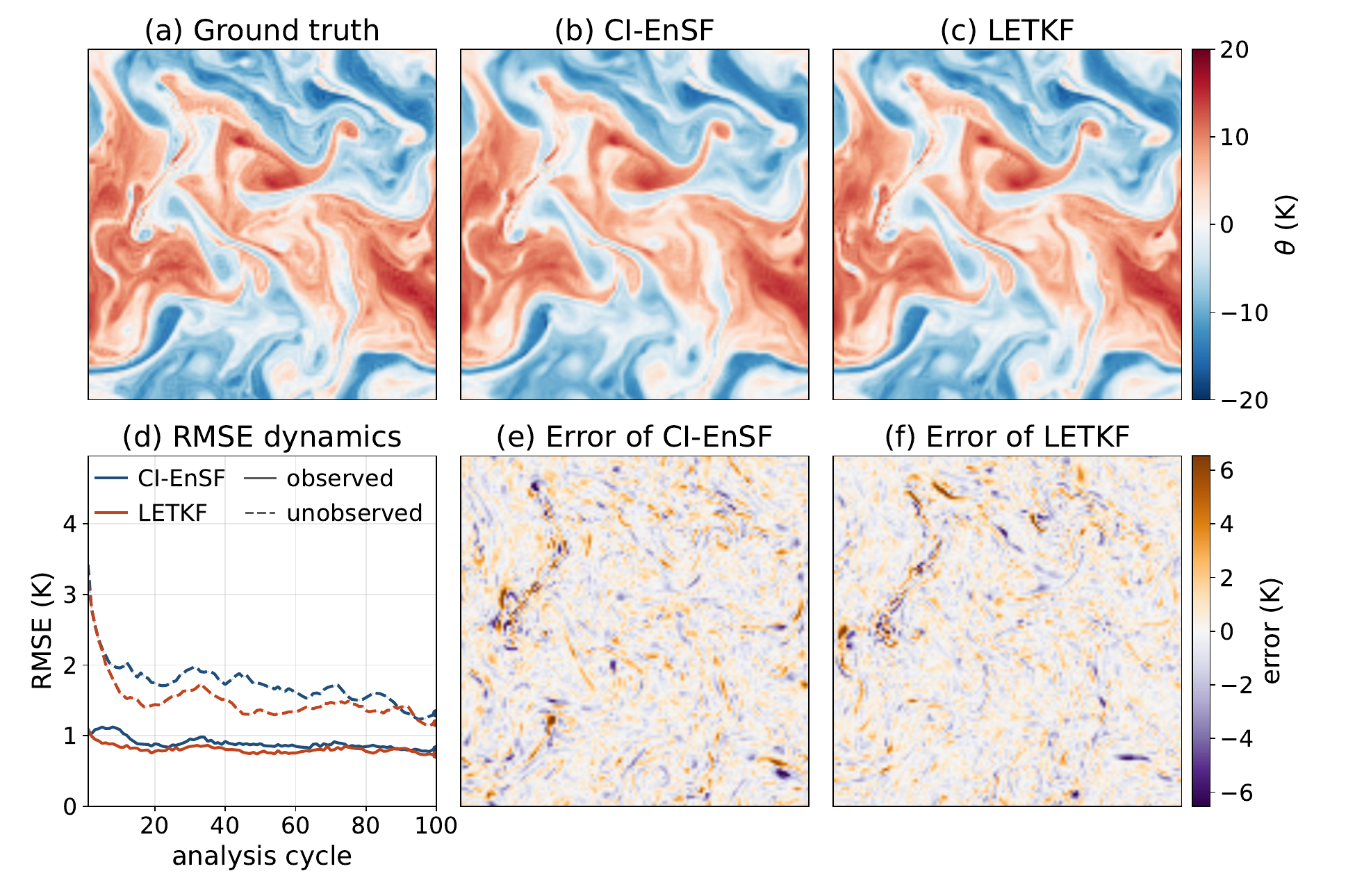}
\vspace{-0.5cm}
\caption{Illustration of CI-EnSF for the SQG model with linear observations.
Panels (a)--(c) show the ground truth, CI-EnSF analysis, and LETKF analysis
at the final cycle of an illustrative evaluation realization,
and panels (e)--(f) show the corresponding analysis-error fields using a
common error color scale. Panel (d) shows the cycling evolution of analysis
RMSE for the directly observed and unobserved state components, with the
evaluation period spanning cycles 50--99.
Both methods rapidly reduce the initial error, remain stable in all ten
evaluation realizations, and maintain control of the much larger unobserved
portion of the state. In this linear-observation regime, however, the tuned
LETKF achieves lower mean analysis RMSE than CI-EnSF,
showing that high dimensionality alone does not create an advantage for
score-based transport.}
\label{fig:sqg128}
\end{figure}

\subsubsection{Nonlinear observations}
\label{sec:exp-sqg128-arctan}

The nonlinear SQG experiment uses the saturating observation operator
\begin{equation}
  h(\theta)
  =
  s\arctan(\theta/s),
  \qquad
  s=2.3~\mathrm{K}.
  \label{eq:sqg_nonlinear_obs}
\end{equation}
Because the observation error is specified in the transformed observation
space, the effective uncertainty expressed in state-space units depends on
the local slope of \(h\). A first-order transformation gives
\begin{equation}
  \sigma_{\rm eff}(\theta)
  =
  \frac{\sigma_{\rm obs}}{|h'(\theta)|}
  =
  \sigma_{\rm obs}
  \left[
    1+\left(\frac{\theta}{s}\right)^2
  \right].
  \label{eq:sqg_sigma_eff}
\end{equation}
If the same value \(\sigma_{\rm obs}=1\)~K as in the linear experiment were
used, the median \(\sigma_{\rm eff}\) along the truth trajectory would be
6.41~K, making the nonlinear observations substantially less informative in
state space. We therefore set \(\sigma_{\rm obs}=0.1559\)~K, for which the
median effective state-space observation error is approximately 1~K. This
normalization prevents the saturation of \(h\) from making the observations
trivially weak. It should not, however, be interpreted as a strictly matched
comparison with the linear SQG experiment, because the assimilation interval
and observation density also differ between the two configurations.

Analyses are performed every 24~h for 150 cycles. At each cycle, 25\% of the
grid points on both surfaces are selected randomly and observed through
Eq.~\eqref{eq:sqg_nonlinear_obs}; the observation network is redrawn at
the next cycle. Verification statistics are computed over cycles 75--149.
Hyperparameter tuning selects \(L=1000\)~km and \(\kappa=0.025\) for
CI-EnSF, and \(L=500\)~km with RTPS \(=0.3\) for LETKF. Table \ref{tab:sensitivity} shows the corresponding
\(4\times6\) sensitivity study which shows that a broader region of stable
parameter choices for CI-EnSF (13 of 24 settings) than for LETKF (6 of 24 settings).
\begin{table}[htbp]
\centering
\caption{Hyperparameter sensitivity comparison. Each entry is the RMSE (K) averaged over the three realizations for tuning.
The notation ``\textrm{---}'' indicates a hyperparameter configuration that overflowed the forecast model on at least one realization; the notation ``$\dagger$'' indicates that the RMSE is above the divergence threshold.}
\label{tab:sensitivity}
\small
\begin{tabular}{ccccccc}
\toprule
\multicolumn{7}{c}{LETKF's hyperparameter tuning chart}\\
\toprule
\diagbox[width=3cm,height=1cm]{L (km)}{RTPS} & $0$ & $0.1$ & $0.3$ & $0.5$ & $0.7$ & $0.9$ \\
\hline
$250$ & $2.867$ & $2.555$ & $2.629$ & --- & --- & --- \\
$500$ & --- & --- & $2.252$ & $2.647$ & --- & --- \\
$1000$ & --- & --- & --- & --- & --- & --- \\
$2000$ & --- & --- & --- & --- & $2.654$ & --- \\
\toprule
\toprule
\multicolumn{7}{c}{CI-EnSF's hyperparameter tuning chart} \\
\toprule
\diagbox[width=3cm,height=1cm]{L (km)}{$\kappa$} & $0$ & $0.0025$ & $0.01$ & $0.025$ & $0.05$ & $0.1$ \\
\midrule
$250$ & $\dagger$ & $\dagger$ & $\dagger$ & $\dagger$ & $\dagger$ & $\dagger$ \\
$500$ & $2.106$ & $2.106$ & $2.103$ & $2.225$ & $2.423$ & $2.781$ \\
$1000$ & $1.753$ & $1.753$ & $1.760$ & $1.738$ & $1.804$ & $1.969$ \\
$2000$ & --- & --- & --- & --- & --- & $2.624$ \\
\bottomrule
\end{tabular}
\end{table}

In contrast to the linear SQG experiment,
Table~\ref{tab:sqg128arctan} shows that CI-EnSF remains stable in all ten independent
evaluation realizations, with a mean RMSE of 1.7603~K and a Spread/RMSE ratio
of 0.97. LETKF remains stable in only six of the ten realizations; over those
stable realizations, its mean RMSE is 2.2191~K and its Spread/RMSE ratio is
0.91. On the six realizations for which both filters remain stable, CI-EnSF
has lower RMSE in every case, with an average reduction of 21.0\% relative to
LETKF. Thus, the advantage of CI-EnSF in this regime appears in both analysis
accuracy and cycling robustness.
Figure~\ref{fig:sqg128arctan} provides spatial context for this difference.
The improvement of CI-EnSF appears not only
after covariance-informed
propagation into unobserved regions, but also
directly at the observed locations where the nonlinear observation operator
enters the analysis. The corresponding error fields show that both methods
retain larger errors near sharp fronts and filamentary structures, but
CI-EnSF produces smaller errors over much of the domain. Together with the linear SQG experiment, these results indicate that
the benefit of CI-EnSF is associated with the nonlinear observation-to-state
relationship rather than with high dimensionality alone.

\begin{table}[htbp]
\centering
\caption{Performance of CI-EnSF and LETKF for the SQG model with
nonlinear observations. RMSE and Spread/RMSE are computed over the realizations that remain
stable through the experiment; the reported range gives the minimum and maximum
realization-mean RMSE among those stable realizations. CI-EnSF remains stable in
all ten evaluation realizations and has lower RMSE, whereas LETKF remains stable
in six of ten realizations.}
\label{tab:sqg128arctan}
\small
\begin{tabular}{ccccc}
\toprule
Method & Stable realizations & RMSE (K) & RMSE range (K) & Spread/RMSE \\
\midrule
CI-EnSF &  $10/10$ &
  $1.7603 $ &  $1.6547$--$1.8657$ & $0.97$ \\
LETKF &  $6/10$ &
  $2.2191 $ & $2.0065$--$2.3961$ & $0.91$ \\
\bottomrule
\end{tabular}
\end{table}

\begin{figure}[htbp]
\centering
\includegraphics[width=\textwidth]{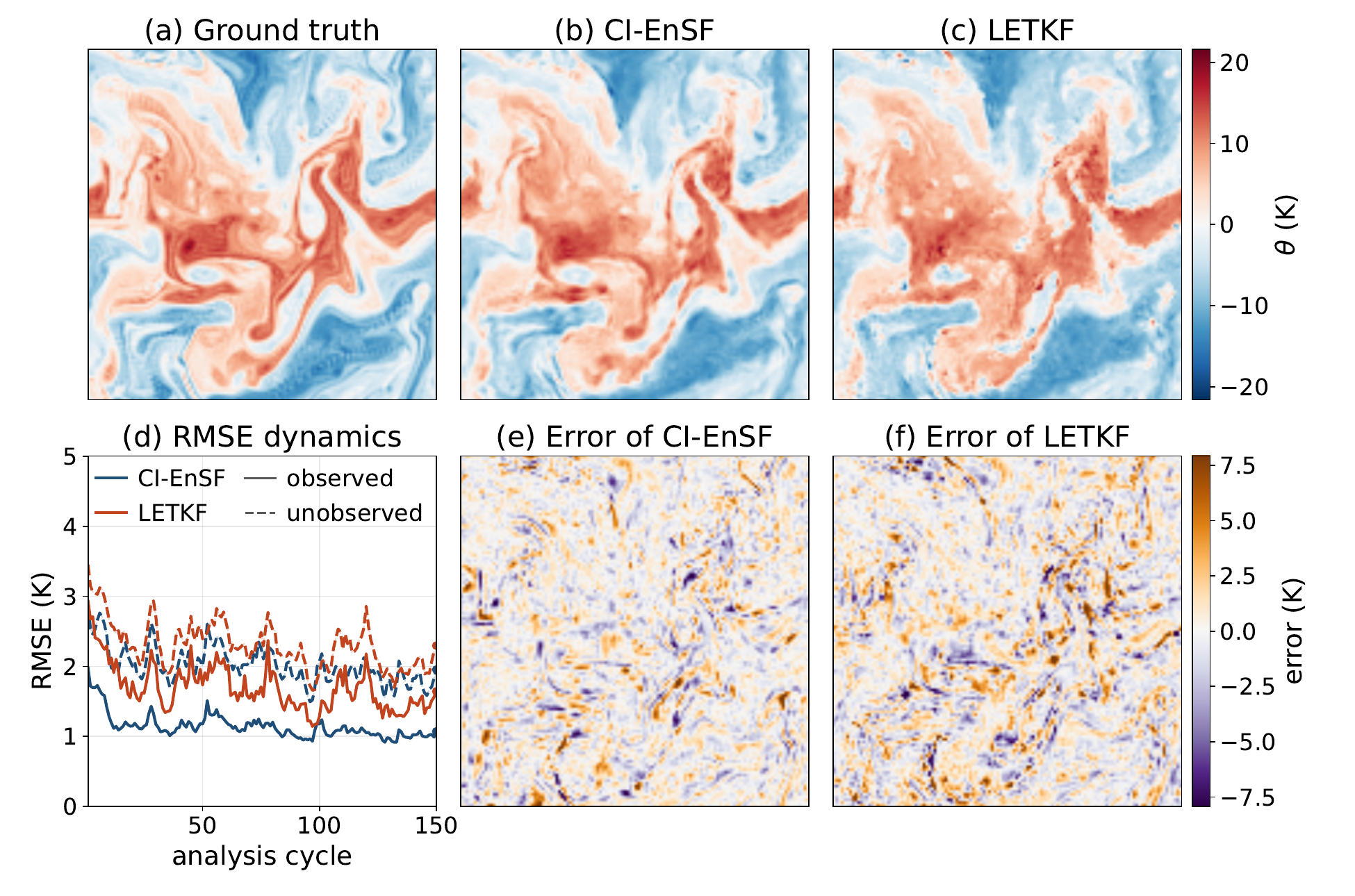}
\vspace{-0.8cm}
\caption{Illustration of CI-EnSF for the SQG model with nonlinear observations.
Panels (a)--(c) show the ground truth, CI-EnSF analysis, and LETKF analysis
at the final cycle of an illustrative evaluation realization, for which both methods remain stable; panels (e)--(f) show
the corresponding analysis-error fields using a common error color scale.
Panel (d) shows the cycling evolution of analysis RMSE for the directly
observed and unobserved state components, with the evaluation period spanning
cycles 75--149. Both methods reduce the initial
climatological error, but CI-EnSF maintains lower error at both observed and
unobserved components during the evaluation period. The spatial error fields
show the same behavior, with smaller CI-EnSF errors over much of the domain, including regions associated with sharp fronts and filamentary structures.}
\label{fig:sqg128arctan}
\end{figure}

\begin{figure}
    \centering    \includegraphics[width=\linewidth]{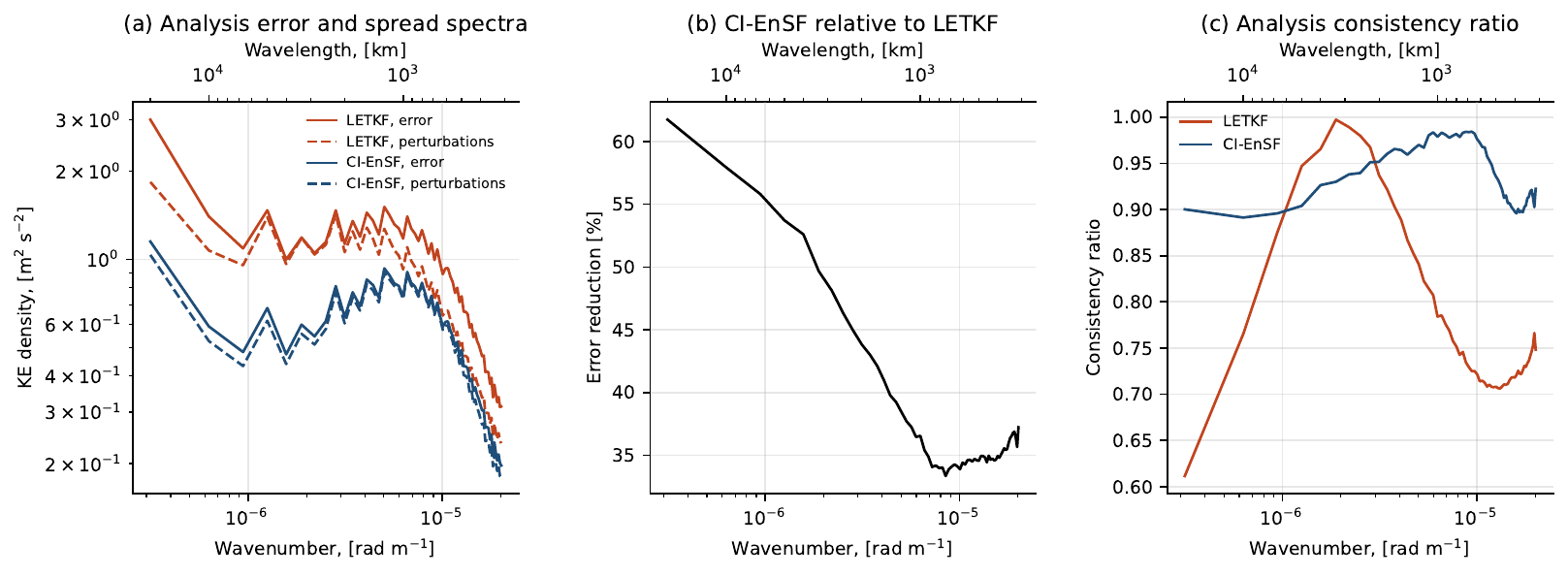}
    \vspace{-0.8cm}
    \caption{Scale decomposition of the analysis error. (a) Kinetic energy of the analysis-mean error (solid) and of
the analysis perturbations (dashed) for each method. The perturbation spectrum is the unbiased ensemble
variance. (b) $100\,(E_{\rm LETKF} - E_{\rm
CI-EnSF})/E_{\rm LETKF}$: the fraction of the LETKF's error energy that CI-EnSF removes, shell
by shell. (c) Consistency ratio, the dashed curve of (a) divided by the solid one; for a
reliable ensemble of $N_e$ members its expected value is $N_e/(N_e+1) = 0.941$.  Spectra are averaged over cycles $75$--$149$ and over the six common stable evaluation realizations.}
\label{fig:spectra}
\end{figure}

Figure~\ref{fig:spectra} complements the bulk RMSE diagnostics by showing how the analysis error and ensemble dispersion vary with spatial scale. CI-EnSF has lower analysis-error energy than LETKF across the resolved wavenumber range, and the relative-error reduction in Figure~\ref{fig:spectra}b remains positive at all scales, indicating that its advantage is not confined to a narrow spectral band. The largest reductions occur at the largest resolved scales, which may partly reflect the different localization lengths selected for the two methods (\(L=1000\)~km for CI-EnSF and \(L=500\)~km for LETKF), although the figure does not isolate localization from differences in the analysis update itself. Figure~\ref{fig:spectra}c further shows that the LETKF perturbation-to-error ratio is strongly scale dependent, whereas the CI-EnSF ratio is more uniform and remains closer to the finite-ensemble consistency reference \(N_e/(N_e+1)=16/17\approx0.941\) over most resolved scales. Thus, CI-EnSF exhibits both lower spectral error and more uniform scale-dependent ensemble reliability in this nonlinear SQG experiment.

\section{Conclusion}
\label{sec:con}
We introduced CI-EnSF, a covariance-informed ensemble score filter that embeds localized
forecast-error covariance geometry directly into score-based transport. The resulting
covariance-shaped diffusion is equivalent to isotropic diffusion in control space, while its
member-conditioned transition kernel yields an analytical prior-score contribution in which the
covariance operator cancels from the dominant prior transport but remains active in the likelihood
transport. This preserves the covariance-mediated pathway by which sparse observations influence
unobserved state components while extending the analysis from a Gaussian covariance update to
nonlinear distributional transport. The numerical experiments show a clear regime dependence.
For linear observations, well-tuned LETKF remains more accurate, while CI-EnSF retains stable
cycling and consistent information propagation. Under nonlinear observations and climatological
initialization, however, CI-EnSF achieves lower analysis error and greater robustness in both
Lorenz--96 and the \(d=32{,}768\) SQG system. In the nonlinear SQG experiment, the improvement
extends across the resolved spectrum and is accompanied by more uniform scale-dependent ensemble
reliability. These results indicate that the main benefit of CI-EnSF is not high dimensionality by
itself, but the combination of nonlinear likelihood geometry with covariance-informed
distributional transport.

Several directions remain for future work. The current formulation uses an approximate
likelihood-score construction and an explicitly integrated reverse ODE, so improved posterior-score
approximations, higher-order or adaptive time integration, and implicit solvers may further reduce
numerical stiffness and computational cost. The innovation-based prior-scaling controller could
also be compared systematically with direct innovation-based covariance-scale estimators, and the
interaction among localization, adaptive prior scaling, and likelihood tempering warrants a more
complete theoretical analysis. More broadly, CI-EnSF should be tested with multivariate,
heterogeneous observing systems, hybrid or physics-informed covariance models, and realistic
atmospheric prediction systems. Such extensions would clarify whether covariance-informed
score transport can provide a practical route for incorporating mature data-assimilation
covariance infrastructure into non-Gaussian ensemble filtering at operational scales.

\section*{Acknowledgments}
This research was supported by the U.S. Department of Energy (DOE), Office of Science, Office of Advanced Scientific Computing Research (ASCR), Applied Mathematics Program under contracts ERKJ443 and ERKJ388.
Feng Bao acknowledges support from the U.S. National Science Foundation (NSF) under Grants DMS-2142672 and DMS-2608431, and from the U.S. Department of Energy, Office of Science, Advanced Scientific Computing Research (ASCR) Applied Mathematics Program, under Grants DE-SC0025412 and DE-SC0026871. Hristo Chipilski acknowledges support from Florida State University’s CRC Seed Grant (047080). Peter Jan van Leeuwen acknowledges support from the National Oceanic and Atmospheric Administration (NOAA) project CADRE (NA24OARX459C0002-T1-01) and the Department of Defense (DOD) project RAM-HORNS (N00014-24-2017).

\section*{Open Research Statement}
The numerical results for CI-EnSF and the reference methods can be re-produced using the code available at \url{https://github.com/gnzhang09/Covariance-Informed-EnSF} and the data available at \url{https://zenodo.org/records/23576865}.

\section*{Conflict of Interest}
The authors declare there are no conflicts of interest for this manuscript.

\bibliography{CVT_EnSF_references}

@article{BaoEtAl2024ScoreFilter,
  author  = {Bao, F. and Zhang, Z. and Zhang, G.},
  title   = {A score-based filter for nonlinear data assimilation},
  journal = {Journal of Computational Physics},
  volume  = {514},
  pages   = {113207},
  year    = {2024},
  doi     = {10.1016/j.jcp.2024.113207}
}

@article{ElGharamti2018,
  author  = {El Gharamti, Mohamad},
  title   = {Enhanced Adaptive Inflation Algorithm for Ensemble Filters},
  journal = {Monthly Weather Review},
  volume  = {146},
  number  = {2},
  pages   = {623--640},
  year    = {2018},
  doi     = {10.1175/MWR-D-17-0187.1}
}

@article{Anderson2009,
  author  = {Anderson, Jeffrey L.},
  title   = {Spatially and temporally varying adaptive covariance inflation for ensemble filters},
  journal = {Tellus A: Dynamic Meteorology and Oceanography},
  volume  = {61},
  number  = {1},
  pages   = {72--83},
  year    = {2009},
  doi     = {10.1111/j.1600-0870.2008.00361.x}
}

@article{XiongEtAl2026ObservationNetworks,
  author  = {Xiong, Z. and Liang, S. and Bao, F. and Zhang, G. and Chipilski, H.},
  title   = {Robustness of the ensemble score filters to the type of assimilated observation networks},
  journal = {Atmospheric Science Letters},
  volume  = {27},
  number  = {1},
  pages   = {e70004},
  year    = {2026},
  doi     = {10.1002/asl.70004}
}

@article{BaoEtAl2026UnitedFilter,
  author  = {Bao, F. and Zhang, Z. and Zhang, G.},
  title   = {United filter for jointly estimating state and parameters of stochastic dynamical systems},
  journal = {Communications in Computational Physics},
  volume  = {39},
  number  = {3},
  pages   = {747--774},
  year    = {2026},
  doi     = {10.4208/cicp.OA-2024-0009}
}

@article{HuynhEtAl2026AdaptiveLearning,
  author  = {Huynh, T. and Fajardo, R. and Zhang, G. and Ju, L. and Bao, F.},
  title   = {A score-based diffusion model approach for adaptive learning of stochastic partial differential equation solutions},
  journal = {Journal of Computational Physics},
  volume  = {556},
  pages   = {114814},
  year    = {2026},
  doi     = {10.1016/j.jcp.2026.114814}
}

@article{Yeh_vanLeeuwen_2026a,
author = {Yeh, Hao-Lun and {van Leeuwen}, Peter Jan},
title = {Unbiased fully nonlinear data assimilation. Part I: The Stochastic Particle Flow Filter},
journal = {Quarterly Journal of the Royal Meteorological Society},
pages = {Accepted},
year = {2026}
}

@article{Yeh_vanLeeuwen_2026b,
author = {Yeh, Hao-Lun and {van Leeuwen}, Peter Jan},
title = {Unbiased fully nonlinear data assimilation. Part II: The Stochastic Particle Flow Smoother},
journal = {Quarterly Journal of the Royal Meteorological Society},
pages = {Accepted},
year = {2026}
}

@article{Hu_vanLeeuwen_2021,
author = {Hu, Chih-Chi and {van Leeuwen}, Peter Jan},
title = {A particle flow filter for high-dimensional system applications},
journal = {Quarterly Journal of the Royal Meteorological Society},
volume = {147},
number = {737},
pages = {2352-2374},
doi = {https://doi.org/10.1002/qj.4028},
year = {2021}
}

@article{Lorenc1986,
  author  = {Lorenc, A. C.},
  title   = {Analysis methods for numerical weather prediction},
  journal = {Quarterly Journal of the Royal Meteorological Society},
  year    = {1986},
  volume  = {112},
  number  = {474},
  pages   = {1177--1194},
  doi     = {10.1002/qj.49711247414}
}

@article{Bannister2008a,
  author  = {Bannister, R. N.},
  title   = {A review of forecast error covariance statistics in atmospheric variational data assimilation. I: Characteristics and measurements of forecast error covariances},
  journal = {Quarterly Journal of the Royal Meteorological Society},
  year    = {2008},
  volume  = {134},
  number  = {637},
  pages   = {1951--1970},
  doi     = {10.1002/qj.339}
}

@article{Bannister2008b,
  author  = {Bannister, R. N.},
  title   = {A review of forecast error covariance statistics in atmospheric variational data assimilation. II: Modelling the forecast error covariance statistics},
  journal = {Quarterly Journal of the Royal Meteorological Society},
  year    = {2008},
  volume  = {134},
  number  = {637},
  pages   = {1971--1996},
  doi     = {10.1002/qj.340}
}

@article{HoutekamerMitchell1998,
  author  = {Houtekamer, P. L. and Mitchell, H. L.},
  title   = {Data assimilation using an ensemble Kalman filter technique},
  journal = {Monthly Weather Review},
  year    = {1998},
  volume  = {126},
  number  = {3},
  pages   = {796--811},
  doi     = {10.1175/1520-0493(1998)126<0796:DAUAEK>2.0.CO;2}
}

@article{HamillSnyder2000,
  author  = {Hamill, T. M. and Snyder, C.},
  title   = {A hybrid ensemble Kalman filter--3D variational analysis scheme},
  journal = {Monthly Weather Review},
  year    = {2000},
  volume  = {128},
  number  = {8},
  pages   = {2905--2919},
  doi     = {10.1175/1520-0493(2000)128<2905:AHEKFV>2.0.CO;2}
}

@article{WhitakerHamill2002,
  author  = {Whitaker, J. S. and Hamill, T. M.},
  title   = {Ensemble data assimilation without perturbed observations},
  journal = {Monthly Weather Review},
  year    = {2002},
  volume  = {130},
  number  = {7},
  pages   = {1913--1924},
  doi     = {10.1175/1520-0493(2002)130<1913:EDAWPO>2.0.CO;2}
}

@article{Hunt2007,
  author  = {Hunt, B. R. and Kostelich, E. J. and Szunyogh, I.},
  title   = {Efficient data assimilation for spatiotemporal chaos: A local ensemble transform Kalman filter},
  journal = {Physica D: Nonlinear Phenomena},
  year    = {2007},
  volume  = {230},
  number  = {1--2},
  pages   = {112--126},
  doi     = {10.1016/j.physd.2006.11.008}
}

@article{Bao2024,
  author  = {Bao, Feng and Zhang, Zezhong and Zhang, Guannan},
  title   = {An ensemble score filter for tracking high-dimensional nonlinear dynamical systems},
  journal = {Computer Methods in Applied Mechanics and Engineering},
  year    = {2024},
  volume  = {432},
  pages   = {117447},
  doi     = {10.1016/j.cma.2024.117447}
}

@article{Liang2025Inpainting,
  author  = {Liang, Siming and Tran, Hoang and Bao, Feng and Chipilski, Hristo G. and van Leeuwen, Peter Jan and Zhang, Guannan},
  title   = {Ensemble score filter with image inpainting for data assimilation in tracking surface quasi-geostrophic dynamics with partial observations},
  journal = {SIAM Journal on Applied Mathematics},
  year    = {2026},
  volume = {86},
  issue = {5},
  pages = {2379-2404},
  doi ={10.1137/25M1730053}
}

@article{Xiong2026EnSFLR,
  author  = {Xiong, Zixiang and Bao, Feng and Chipilski, Hristo G. and Liang, Siming and Tang, Jingqiao and Zhang, Guannan},
  title   = {A Two-Step Ensemble Score Filter for Data Assimilation in Partially Observed Systems},
  journal = {arXiv preprint arXiv:2606.28264},
  year    = {2026},
  eprint  = {2606.28264},
  archivePrefix = {arXiv}
}

@article{zhang2026iensf,
  author  = {Zhang, Z. and Bao, F. and Zhang, G.},
  title   = {{IEnSF}: Iterative Ensemble Score Filter for Reducing Error in Posterior Score Estimation in Nonlinear Data Assimilation},
  journal = {Journal of Computational Physics},
  volume  = {568},
  pages   = {115357},
  year    = {2026},
  doi     = {10.1016/j.jcp.2026.115357}
}

@article{DeeDaSilva1999,
  author  = {Dee, Dick P. and da Silva, Arlindo M.},
  title   = {Maximum-Likelihood Estimation of Forecast and Observation Error Covariance Parameters. Part I: Methodology},
  journal = {Monthly Weather Review},
  year    = {1999},
  volume  = {127},
  number  = {8},
  pages   = {1822--1834}
}

@article{DesroziersEtAl2005,
  author  = {Desroziers, G{\'e}rald and Berre, Lo{\"\i}k and Chapnik, Boris and Poli, Paul},
  title   = {Diagnosis of Observation, Background and Analysis-Error Statistics in Observation Space},
  journal = {Quarterly Journal of the Royal Meteorological Society},
  year    = {2005},
  volume  = {131},
  number  = {613},
  pages   = {3385--3396}
}

@article{Anderson2007,
  author  = {Anderson, Jeffrey L.},
  title   = {An Adaptive Covariance Inflation Error Correction Algorithm for Ensemble Filters},
  journal = {Tellus A: Dynamic Meteorology and Oceanography},
  year    = {2007},
  volume  = {59},
  number  = {2},
  pages   = {210--224}
}

@article{SakovEtAl2012,
  author  = {Sakov, Pavel and Oliver, Dean S. and Bertino, Laurent},
  title   = {An Iterative {EnKF} for Strongly Nonlinear Systems},
  journal = {Monthly Weather Review},
  year    = {2012},
  volume  = {140},
  number  = {6},
  pages   = {1988--2004},
  doi     = {10.1175/MWR-D-11-00176.1}
}

@article{EmerickReynolds2013,
  author  = {Emerick, Alexandre A. and Reynolds, Albert C.},
  title   = {Ensemble Smoother with Multiple Data Assimilation},
  journal = {Computers \& Geosciences},
  year    = {2013},
  volume  = {55},
  pages   = {3--15},
  doi     = {10.1016/j.cageo.2012.03.011}
}

@article{VanLeeuwenEtAl2019,
  author  = {van Leeuwen, Peter Jan and K{\"u}nsch, Hans R. and
             Nerger, Lars and Potthast, Roland and Reich, Sebastian},
  title   = {Particle Filters for High-Dimensional Geoscience Applications:
             A Review},
  journal = {Quarterly Journal of the Royal Meteorological Society},
  year    = {2019},
  volume  = {145},
  number  = {723},
  pages   = {2335--2365},
  doi     = {10.1002/qj.3551}
}

@article{KurosawaPoterjoy2021,
  author  = {Kurosawa, Kenta and Poterjoy, Jonathan},
  title   = {Data Assimilation Challenges Posed by Nonlinear Operators:
             A Comparative Study of Ensemble and Variational Filters and
             Smoothers},
  journal = {Monthly Weather Review},
  year    = {2021},
  volume  = {149},
  number  = {7},
  pages   = {2369--2389},
  doi     = {10.1175/MWR-D-20-0368.1}
}

@article{Chipilski2025,
  author  = {Chipilski, Hristo G.},
  title   = {Exact Nonlinear State Estimation},
  journal = {Journal of the Atmospheric Sciences},
  year    = {2025},
  volume  = {82},
  number  = {4},
  pages   = {809--827},
  doi     = {10.1175/JAS-D-24-0171.1}
}

@article{GaspariCohn1999,
  author  = {Gaspari, Gregory and Cohn, Stephen E.},
  title   = {Construction of correlation functions in two and three dimensions},
  journal = {Quarterly Journal of the Royal Meteorological Society},
  year    = {1999},
  volume  = {125},
  number  = {554},
  pages   = {723--757},
  doi     = {10.1002/qj.49712555417}
}

@article{HamillEtAl2001,
  author  = {Hamill, Thomas M. and Whitaker, Jeffrey S. and Snyder, Chris},
  title   = {Distance-Dependent Filtering of Background Error Covariance Estimates in an Ensemble {Kalman} Filter},
  journal = {Monthly Weather Review},
  year    = {2001},
  volume  = {129},
  number  = {11},
  pages   = {2776--2790},
  doi     = {10.1175/1520-0493(2001)129<2776:DDFOBE>2.0.CO;2}
}

@article{HoutekamerZhang2016,
  author  = {Houtekamer, P. L. and Zhang, Fuqing},
  title   = {Review of the Ensemble {Kalman} Filter for Atmospheric Data Assimilation},
  journal = {Monthly Weather Review},
  year    = {2016},
  volume  = {144},
  number  = {12},
  pages   = {4489--4532},
  doi     = {10.1175/MWR-D-15-0440.1}
}

@article{Lorenc2003,
  author  = {Lorenc, Andrew C.},
  title   = {The potential of the ensemble {Kalman} filter for {NWP}---a comparison with {4D-Var}},
  journal = {Quarterly Journal of the Royal Meteorological Society},
  year    = {2003},
  volume  = {129},
  number  = {595},
  pages   = {3183--3203},
  doi     = {10.1256/qj.02.132}
}

@article{Buehner2005,
  author  = {Buehner, Mark},
  title   = {Ensemble-derived stationary and flow-dependent background-error covariances: Evaluation in a quasi-operational {NWP} setting},
  journal = {Quarterly Journal of the Royal Meteorological Society},
  year    = {2005},
  volume  = {131},
  number  = {607},
  pages   = {1013--1043},
  doi     = {10.1256/qj.04.15}
}

@article{BocquetEtAl2010,
  author  = {Bocquet, Marc and Pires, Carlos A. and Wu, Lin},
  title   = {Beyond {Gaussian} Statistical Modeling in Geophysical Data Assimilation},
  journal = {Monthly Weather Review},
  year    = {2010},
  volume  = {138},
  number  = {8},
  pages   = {2997--3023},
  doi     = {10.1175/2010MWR3164.1}
}

@book{KalnayEtAl2024,
  author    = {Kalnay, Eugenia and Mote, Safa and Da, Cheng},
  title     = {Earth System Modeling, Data Assimilation and Predictability:
               Atmosphere, Oceans, Land and Human Systems},
  edition   = {2},
  publisher = {Cambridge University Press},
  year      = {2024},
  doi       = {10.1017/9780511920608},
  isbn      = {9781107009004}
}

@book{AschEtAl2016,
  author    = {Asch, Mark and Bocquet, Marc and Nodet, Ma{\"e}lle},
  title     = {Data Assimilation: Methods, Algorithms, and Applications},
  series    = {Fundamentals of Algorithms},
  volume    = {11},
  publisher = {Society for Industrial and Applied Mathematics},
  address   = {Philadelphia, PA},
  year      = {2016},
  doi       = {10.1137/1.9781611974546},
  isbn      = {9781611974539}
}

@book{EvensenEtAl2022,
  author    = {Evensen, Geir and Vossepoel, Femke C. and van Leeuwen, Peter Jan},
  title     = {Data Assimilation Fundamentals: A Unified Formulation of the State and Parameter Estimation Problem},
  series    = {Springer Textbooks in Earth Sciences, Geography and Environment},
  publisher = {Springer},
  address   = {Cham},
  year      = {2022},
  doi       = {10.1007/978-3-030-96709-3},
  isbn      = {9783030967086}
}

\begin{appendices}
\renewcommand{\theequation}{\thesection\arabic{equation}}
\counterwithin*{equation}{section}

\section{The proof of Proposition \ref{prop2}}\label{sec:app1}

\begin{proof}
From the conditional transition law, we have
\[
    \zb_t \mid \zb_0
    \sim
    \mathcal{N}\!\left((1-t)\zb_0,\; t\Cmat\right),
\]
which is equivalent to
$
    \zb_t = (1-t)\zb_0 + \eta_t,
$
where
$
    \eta_t \sim \mathcal{N}(0,t\Cmat)
$
is independent of $\zb_0$ conditional on $\ub_i$. Since
\[
    \zb_0 \mid \ub_i
    \sim
    \mathcal{N}(\ub_i,\gamma \Cmat),
\]
the pair $(\zb_0,\eta_t)$ is jointly Gaussian conditional on $\ub_i$.
Therefore $\zb_t$, being an affine transformation of jointly Gaussian
random variables, is also Gaussian conditional on $\ub_i$. It remains
to compute its conditional mean and covariance.

By the law of total expectation, we have
\begin{align*}
    \mathbb{E}_{\zb_t|\ub_i}[\zb_t \mid \ub_i]
    &=
    \mathbb{E}_{\zb_0|\ub_i}
    \left[
        \mathbb{E}_{\zb_t|\zb_0,\ub_i}[\zb_t \mid \zb_0,\ub_i]
        \,\middle|\, \ub_i
    \right] \\
    &=
    \mathbb{E}_{\zb_0|\ub_i}
    \left[
        (1-t)\zb_0
        \,\middle|\, \ub_i
    \right] \\
    &=
    (1-t)\mathbb{E}_{\zb_0|\ub_i}[\zb_0 \mid \ub_i] \\
    &=
    (1-t)\ub_i.
\end{align*}
The conditional covariance follows from the law of total covariance:
\begin{align*}
    \operatorname{Cov}(\zb_t \mid \ub_i)
    &=
    \mathbb{E}_{\zb_0|\ub_i}
    \left[
        \operatorname{Cov}(\zb_t \mid \zb_0,\ub_i)
        \,\middle|\, \ub_i
    \right]
    +
    \operatorname{Cov}_{\zb_0|\ub_i}
    \left(
        \mathbb{E}[\zb_t \mid \zb_0,\ub_i]
        \,\middle|\, \ub_i
    \right).
\end{align*}
The first term is
\[
    \mathbb{E}_{\zb_0|\ub_i}
    \left[
        \operatorname{Cov}(\zb_t \mid \zb_0,\ub_i)
        \,\middle|\, \ub_i
    \right]
    =
    t\Cmat,
\]
because the conditional covariance of the forward transition does not
depend on $\zb_0$. For the second term, we have
\begin{align*}
    \operatorname{Cov}_{\zb_0|\ub_i}
    \left(
        \mathbb{E}[\zb_t \mid \zb_0,\ub_i]
        \,\middle|\, \ub_i
    \right)
    =
    \operatorname{Cov}
    \left(
        (1-t)\zb_0
        \,\middle|\, \ub_i
    \right)
    =
    (1-t)^2
    \operatorname{Cov}(\zb_0 \mid \ub_i)
    =
    (1-t)^2 \gamma \Cmat .
\end{align*}
Consequently,
\begin{align*}
    \operatorname{Cov}(\zb_t \mid \ub_i)
    =
    t\Cmat + (1-t)^2\gamma \Cmat
    =
    \left[t+(1-t)^2\gamma\right]\Cmat
    &=
    c(t)\Cmat.
\end{align*}

Combining the Gaussianity established above with the conditional mean
and covariance yields
\[
    \zb_t \mid \ub_i
    \sim
    \mathcal{N}
    \left(
        (1-t)\ub_i,\;
        \left[(1-t)^2\gamma+t\right]\Cmat
    \right),
\]
which proves the result.
\end{proof}

\section{Continuous-time limit of the Kalman update}
\label{app:kalmanbucy}

This appendix derives the continuous-time limit used in
Section~\ref{sec:classicalrelation}. The purpose is to show how the
covariance-weighted observation-to-state operator
\(\Bmat^f\Hmat^{\top}\Rmat_t^{-1}\)
arises from a sequence of discrete Kalman updates as the observation interval
approaches zero.
Consider the continuous-time observation process
\begin{equation}
  d{\bf Y}_t
  =
  \Hmat_t\xb_t\,dt
  +
  \Vmat_t\,d{\bf W}_t^{\rm obs},
  \label{eq:cont_meas}
\end{equation}
where \({\bf W}_t^{\rm obs}\) is a standard observation-space Wiener process.
The continuous-time observation-noise covariance-rate matrix is defined by
\begin{equation}
  \Rmat_t
  =
  \Vmat_t\Vmat_t^{\top}.
  \label{eq:cont_R}
\end{equation}

To connect the continuous observation process with the usual discrete-time
data-assimilation model, consider an observation interval
\(\Delta t=t_{k+1}-t_k\). The interval-averaged observation is defined by
\begin{equation}
  \yb_k
  =
  \frac{{\bf Y}_{t_{k+1}}-{\bf Y}_{t_k}}{\Delta t}.
  \label{eq:discrete_average}
\end{equation}
Integration of the observation process in Eq.~\eqref{eq:cont_meas} over the
interval gives
\begin{equation}
  \yb_k
  =
  \frac{1}{\Delta t}
  \int_{t_k}^{t_{k+1}}
  \Hmat_t\xb_t\,dt
  +
  \eb_k,
  \label{eq:discrete_from_cont}
\end{equation}
where the corresponding discrete observation-noise term is
\begin{equation}
  \eb_k
  =
  \frac{1}{\Delta t}
  \int_{t_k}^{t_{k+1}}
  \Vmat_t\,d{\bf W}_t^{\rm obs}.
  \label{eq:discrete_noise}
\end{equation}
When \(\Hmat_t\) and \(\xb_t\) vary negligibly over an infinitesimal interval,
the deterministic part of Eq.~\eqref{eq:discrete_from_cont} reduces to
\(\Hmat_k\xb_k\), giving the standard discrete observation model
\begin{equation}
  \yb_k
  =
  \Hmat_k\xb_k
  +
  \eb_k.
  \label{eq:discrete_meas}
\end{equation}

The discrete observation-error covariance matrix is defined by
\begin{equation}
  \Rmat_d
  =
  \mathbb{E}
  \left[
    \eb_k\eb_k^{\top}
  \right].
  \label{eq:Rd_definition}
\end{equation}
Using the It\^o isometry together with the noise representation in
Eq.~\eqref{eq:discrete_noise}, the covariance becomes
\begin{equation}
  \Rmat_d
  =
  \frac{1}{(\Delta t)^2}
  \int_{t_k}^{t_{k+1}}
  \Vmat_t\Vmat_t^{\top}\,dt
  =
  \frac{1}{(\Delta t)^2}
  \int_{t_k}^{t_{k+1}}
  \Rmat_t\,dt.
  \label{eq:Rd_integral}
\end{equation}
If the covariance-rate matrix varies negligibly over the infinitesimal
interval, the discrete covariance satisfies
\begin{equation}
  \Rmat_d
  =
  \frac{\Rmat_t}{\Delta t}
  +
  o\!\left(\frac{1}{\Delta t}\right).
  \label{eq:Rd_scaling}
\end{equation}
Thus, the observation-error covariance of the interval-averaged observation
increases as \(1/\Delta t\) as the observation interval decreases.

For a forecast covariance \(\Bmat^f\), the discrete Kalman gain associated
with the observation model in Eq.~\eqref{eq:discrete_meas} is
\begin{equation}
  \Kmat_{\rm discrete}
  =
  \Bmat^f\Hmat^{\top}
  \left(
    \Hmat\Bmat^f\Hmat^{\top}
    +
    \Rmat_d
  \right)^{-1}.
  \label{eq:K_discrete_app}
\end{equation}
Substitution of the covariance scaling in Eq.~\eqref{eq:Rd_scaling} gives,
to leading order in \(\Delta t\),
\begin{equation}
  \Kmat_{\rm discrete}
  =
  \Delta t\,
  \Bmat^f\Hmat^{\top}
  \left(
    \Delta t\,
    \Hmat\Bmat^f\Hmat^{\top}
    +
    \Rmat_t
  \right)^{-1}.
  \label{eq:K_scaled_app}
\end{equation}
The Kalman gain per unit physical time is therefore
\begin{equation}
  \frac{\Kmat_{\rm discrete}}{\Delta t}
  =
  \Bmat^f\Hmat^{\top}
  \left(
    \Delta t\,
    \Hmat\Bmat^f\Hmat^{\top}
    +
    \Rmat_t
  \right)^{-1}.
  \label{eq:K_rate_app}
\end{equation}
The continuous-time limit of this gain rate is
\begin{equation}
  \Kmat_t
  =
  \lim_{\Delta t\rightarrow0}
  \frac{\Kmat_{\rm discrete}}{\Delta t}
  =
  \Bmat^f\Hmat^{\top}\Rmat_t^{-1}.
  \label{eq:K_continuous_app}
\end{equation}

The limiting gain in Eq.~\eqref{eq:K_continuous_app} has the same
covariance-weighted observation-to-state geometry as the instantaneous
CI-EnSF likelihood velocity in Eq.~\eqref{eq:cvtinstant}. In particular,
the factor \(\Bmat^f\Hmat^{\top}\) provides the pathway through which
observation-space information is transferred to the model state, including
unobserved components coupled through forecast-error cross-covariances.

\end{appendices}

\end{document}